\documentclass[1p,11pt]{amsart}

\usepackage{amsthm,amssymb,latexsym,amsmath}
\usepackage[all]{xypic}
\usepackage{dsfont}
\usepackage{color}

\usepackage[symbol]{footmisc}

\newcommand{\calf}{{\mathcal F}}

\newcommand{\catA}[1]{{\mathfrak A}}
\newcommand{\catI}[1]{{\mathfrak I}}
\newcommand{\catS}[1]{{\mathfrak S}}

\newcommand{\qsigma}{{\mathcal Q}_{\Sigma}}
\newcommand{\calq}{\mathcal{Q}}

\newcommand{\p}[1]{{\mathbb{P}^{#1}}}

\newcommand{\pn}{{\mathbb{P}^n}}

\newcommand{\bp}[1]{\widetilde{\mathbb{P}^{#1}}}

\newcommand{\dimg}{\pi_{\ast}}
\newcommand{\pull}{\pi^{\ast}}

\DeclareMathOperator{\codim}{{\rm codim}}
\DeclareMathOperator{\Det}{{\rm det}}

\DeclareMathOperator{\ext}{{\rm Ext}}
\DeclareMathOperator{\Hom}{{\rm Hom}}
\DeclareMathOperator{\phom}{\mathbb{P}{\rm Hom}}
\DeclareMathOperator{\ph}{\mathbb{P}{\rm H}}

\DeclareMathOperator{\ho}{H}

\DeclareMathOperator{\coker}{coker}

\DeclareMathOperator{\ch}{{\rm ch}}
\DeclareMathOperator{\pic}{{\rm Pic}}
\DeclareMathOperator{\ft}{\widetilde{\mathcal{F}}}

\newtheorem{theorem}{Theorem}[section]

\newtheorem{proposition}[theorem]{Proposition}
\newtheorem{lemma}[theorem]{Lemma}
\newtheorem{corollary}[theorem]{Corollary}

\newtheorem{remark}[theorem]{Remark}
\newtheorem{example}[theorem]{Example}
\newtheorem{definition}[theorem]{{\bf Definition}}

\begin{document}

\title{Torsion free instanton sheaves on the blow-up of $\mathbb{P}^{3}$ at a point}
\footnotetext[1]{The author was partially supported by the CAPES-COFECUB $08/ 2018$ project and MATH-AMSUD project: GS$\&$MS $21$-MATH$-06$  $2021\hspace{0.1cm}\& \hspace{0.1cm}2022.$
}

\author{Abdelmoubine Amar Henni $^{\ast}$}



\vspace{1cm}

\begin{abstract}
We consider an extension of the instanton bundles definition, given by Casnati-Coskun-Genk-Malaspina, for Fano threefolds, in order to include non locally-free ones on the blow-up $\widetilde{\mathbb{P}^{3}},$ of the projective $3$-space at a point. With the proposed definition, we prove that any reflexive instanton sheaf must be locally free, and that the strictly torsion free instanton sheaves have singularities of pure dimension $1.$ We construct examples and study their $\mu$-stability. Furthermore, these sheaves will play a role in (partially) compactifying the t'Hooft component of the moduli space of instantons, on $\widetilde{\mathbb{P}^{3}}.$ Finally, examples of these are shown to be smooth and smoothable.  

\end{abstract}


\vspace{1cm}

\maketitle

\tableofcontents

\vspace{1cm}

\section{Introduction}
\label{IN}








The study of instanton bundles, and their Moduli, in the case of Fano threefolds, beside the classical, and very important, case of the projective space, started almost a decade ago by Kuznetsov \cite{kuznetsov} and Faenzi \cite{faenzi}, for Fano varieties of index $1,$ and $2.$ Since then, many examples have been studied \cite{MMJ, AM, ccgm, Sanna}, with special attension to the construction of their moduli and features that generalizes the classical instantons on $\p3,$ such as their splitting behaviour on line.
Recently, Casnati-Coskun-Genk-Malaspina, in \cite{ccgm}, proposed a more general definition of instanton bundles that unifies all the preceding ones. Moreover, in the example of the blow-up, $\bp3,$ of the projective space they give a detailed study of such bundles and their moduli. 
For $X$ a Fano threefold with anti-canonical divisor $-K_{X}$ and index $i_{X},$ let ${\rm q}_{X}:=[\frac{i_{X}}{2}],$ and denote with $h:=\frac{-K_{X}}{i_{X}}$ its fundamental divisor (See \cite{Isk1} for more details ).  

\vspace{0.4cm}

 \begin{definition}\label{definition-0}\cite[Definition 1.1]{ccgm}
 A vector bundle $\mathcal{E}$ of rank $2,$ on a  polarized variety$(X,\mathcal{O}_{X}(h))$ is called an instanton bundle if the following properties hold:
 \begin{itemize}
     \item The first Chern class of $\mathcal{E}$ is $c_{1}(\mathcal{E})=(2{\rm q}_{X}-i_{X})h;$
     \item $\mathcal{E}$ is $\mu$-semi stable, with respect to the polarization $\mathcal{O}_{X}(h),$ given the ample divisor $h;$ 
     \item The cohomology group $\ho^{1}(X,\mathcal{E}(-{\rm q}_{X}h))$ is trivial.
 \end{itemize}
 Moreover, the second chern class $c_{2}(\mathcal{E})$ will be refered to as charge of the instanton bundle $\mathcal{E}.$
 \end{definition}

\vspace{0.4cm} 
 
The definition above is also independent of the rank of the Picard group and happen to be very useful for obtaining a monad-like description of these bundles \cite{kuznetsov, faenzi,ccgm, MMJ}. Monad descriptions are, in most situations, very dependent of the knowledge of the bounded derived category of coherent sheaves $D^{\flat}(X),$ on $X,$ and conditions such as in the second item of the definition above are practical in order to obtain crucial vanishings, so that Beilinson-like spectral sequences degenerate, and consequently leading to a monad. 

The drawback of the definition above is the fact that it does not include non locally free sheaves, neither higher rank cases. In particular, a simple form of Serre duality is enough to get many of the desired vanishings.  
In the present paper we shall extend Definition \ref{definition-0}, on the blow-up $\bp3$ of the projective $3$-space, by adding more cohomology vanishing, in order to include non locally free instantons that can be used compactify the moduli space. In the case of $\p3$ instanton sheaves, such compactifications were also used to understand the geography of the Maruyama moduli space of stable rank $2$ sheaves and related constructions, as in \cite{AJT} and references therein. 
The present paper addresses the non locally free situation, that has been also generalized to blow ups of higher dimensional projective spaces, in a more recent study  \cite{henni1}.  We expect to include wider classes of varieties and tackle higher rank cases, in future works.

In the current circumstance, It also follows, from our definition, that any reflexive instanton sheaf of rank $2$ on $\bp3$ is locally free, alike the case of rank $2$ instantons on $\p3$ \cite{JG} . More precisely,

\vspace{0.4cm}
 
${\bf 1^{{\rm st}}}$\textbf{Main Result:} [Corollary \ref{ccgm-like-monad} and Theorem \ref{main1}] {\em
\begin{itemize}
    \item[(a)] Any torsion free instanton sheaf $\mathcal{F},$ on $\bp3,$ is also the cohomology of a monad: 
    \begin{equation}
0\to\begin{array}{c} \mathcal{O}_{\bp3}(-1,1)^{\oplus (l+\gamma)} \\ \oplus \\ \Omega^{1}(0,-1)^{\oplus (k-l)} \end{array} \stackrel{\alpha}{\to} 
\begin{array}{c} \mathcal{O}_{\bp3}(-1,1)^{\oplus\gamma } \\ \oplus \\ \Omega^{1} (1,-1)^{\oplus k} \\ \oplus \\ \mathcal{O}_{\bp3}(-1,0)^{\oplus 2(k-l)} \end{array} 
\stackrel{\beta}{\to}\mathcal{O}_{\bp3}^{\oplus (2k-l-r)}\to0. 
\end{equation}
In which, the map $\alpha$ fails to be injective exactly on the singularity locus of $\mathcal{F}$.
    \item[(b)] The singularity locus of a strictly torsion free instanton sheaf, on $\bp3$ has pure dimension $1.$ 
\end{itemize}
}
\vspace{0.4cm}

The property in item $({\rm b}),$ are special to rank $2$ and is a generalization of the already known phenomenon on $\p3,$ as it was shown in \cite{JG}, for all rank $2$ instanton sheaves, and in independently in \cite[Theorem 3.4]{henni}, for the special case of the fixed instantons sheaves of rank $2$, with respect to a toric action. This might hold for rank $2$ instantons on many  $3$-folds and is worth of further investigations. For higher ranks this not true and there are counter-examples already on $\p3$ \cite{JG}.

As an application we study the boundary of the component of the instanton bundles given by \cite[Construction 7.2]{ccgm}, and that we will call t'Hooft bundles:
In order to achieve the desired result, we use elementary transformations, as in \cite{JMT1, MT}, with respect to a triple $(\Sigma,\mathcal{L},\phi),$ satisfying certain conditions, where $\Sigma$ is a curve, $\mathcal{L}$ is an invertible sheaf on $\Sigma$ and $\phi$ is a surjective map. In particular, it will follow that $\mu$-(semi-)stability is preserved under these operations. We call a sheaf $\mathcal{F}$ {\em smoothable} if it can be deformed to a locally free sheaf. That is, there is a pointed scheme $(S,s_{0}),$ an open subscheme $U\owns s_{0},$ of $S,$ and a family of sheaves $\mathbb{F}$ parameterized by $S$ such that the central fiber $\mathbb{F}_{s_{0}}=\mathcal{F}$ and for all $u\in U\backslash\{s_{0}\},$ the generic fiber $\mathbb{F}_{u}$ is locally free.

Furthermore if $[F]$ is a point, representing $\mathcal{F},$ in a moduli $\mathcal{M},$ then by smoothable class $[\mathcal{F}]$ we mean that $\mathcal{F}$ is smoothable and $[\mathbb{F}_{u}]\in \mathcal{M}.$ for some $(U',s_{0})\subset (U',s_{0}).$ In the last part of this work we will prove 

\vspace{0.4cm}
${\bf 2^{{\rm nd}}}$\textbf{Main Result:} [Proposition \ref{main2}]
{\em Let $\mathcal{F}$ be the elementary transform of a t'Hooft instanton bundle, on $\bp3,$ with respect to an elementary data $(\Sigma,\mathcal{L},\phi),$ in which $\Sigma$ is a union of lines. Then its isomorphism class $[\mathcal{F}]$ is smooth and smoothable.}
\vspace{0.4cm}

In particular, the isomorphism class of such sheaves does not lie in the intersection of the t'Hooft component with another one, inside the larger moduli of all instantons. 

We observe that the latter result helps us understand only a part the boundary of the t'Hooft component. In the future, We hope to be able to understand the whole boundary and also how other components fit in the larger moduli of all instantons, on $\bp3.$ The relation with the more general moduli of $\mu$-stable sheaves, on $\bp3,$ in general is also of a high interest.







This paper is organized as the following; In section \ref{blow-up} we recall some results about $\bp3$ and set notations. Section \ref{instanton} will be dedicated to our extended definition of instanton sheaves and its consequences. In particular, we give the proof of our first main Theorem \ref{main1}. In Section \ref{elementary}, we introduce elementary transformations of instantons and study $\mu$-stability of the obtained sheaves, with respect to these operations. In particular we show the existence of torsion free instanton sheaves with the proposed Definition \ref{Instanton-def}. Finally, in Section \ref{application} we prove smothness and smoothability of elementary transformations of t'Hooft instanton bundles on $\bp3.$ 


\section{The Blow-up of the projective space at a point}\label{blow-up}

Let $\p3$ be the projective $3$-space, over the field of complex numbers $\mathbb{C},$ and fix a point $p_{0}$ in it. We will denote by $\bp3$ be the blow-up of $\p3$ at $p_{0},$ and by $\pi:\bp3\to\p3$ the blow-down map. Let $H,$ be the pull-back of the cohomology class of plane section, in $\ho^{2}(\p3,\mathbb{Z}),$ and set $E$ for the cohomology class of the exceptional divisor in $\ho^{2}(\bp3,\mathbb{Z}),$ then the Picard group of $\bp3$ can be written as ${\rm Pic}(\bp3)=H\mathbb{Z}\oplus E\mathbb{Z},$ 
and one can show that the Chow ring of $\bp3$ is isomorphic to the ring $A^{\ast}(\bp3)=\frac{\mathbb{Z}[E,H]}{<E\cdot H, E^{3}-H^{3}>}.$

For a divisor $D=a\cdot H+b\cdot E$ we associate the the line bundle $\mathcal{O}(D)=\mathcal{O}(a,b):=\mathcal{O}(a\cdot H)\otimes\mathcal{O}(b\cdot E).$ In particular, the canonical bundle, associated to the canonical divisor ${\rm K}_{\bp3}=-4H+2E,$ is denoted by $\omega_{\bp3}=\mathcal{O}(-4,2).$ We recall that the Hirzebruch-Riemann-Roch formula for coherent sheaves $\mathcal{F}$ and $\mathcal{G},$ on a smooth variety $X$ is given by $$\chi(F,G):=\sum_{i=0}^{{\rm dim}X}(-1)^{i}{\rm dim}\ext^{i}(F,G)=\int \ch^{\ast}(F)\cdot\ch(G)\cdot{\rm Td}(X),$$
where ${\rm Td}(X)$ is the Todd class of $X$ and the dual Chern character ${\rm ch}^{\ast}(F)$ is given by $$\oplus_{i}(-)^{i}{\rm ch}_{i}(F)\in\ho^{2\bullet}(X,\mathbb{Z}).$$

%
%
%
%
%
%
%
%
%
%
%
%

For a line bundle this reads $$\chi(\mathcal{O}(p,q))=\frac{1}{6}((p+1)(p+2)(p+3)+q(q-1)(q-2))$$
More generally, the Chern character  of coherent sheaf $\mathcal{F}$ can be written as 
\begin{align}\label{chern-ch}
{\rm ch}(F)=r+[aH+bE]&+\frac{1}{2}[(a^{2}-2k)H^{2}+(b^{2}-2l)E^{2}] \notag \\
&+\frac{1}{6}(a^{3}+b^{3}-3ak-3bl+3m)[H^{3}],
\end{align}
where $l,k,m$ are, respectively, integer multiples of the generators $E^{2},$ $H^{2}$ in $A^{2}(\bp3)$ and $H^{3}$ in $A^{3}(\bp3)$.
Specifically, they apear in the Chern classes of $\mathcal{F},$ see \cite{EH}, for details;
\begin{align}
c_{1}(\mathcal{F})&=a \cdot H + b \cdot E \notag  \\
c_{2}(\mathcal{F})&= k \cdot H^{2} + l \cdot E^{2}\notag  \\
c_{3}(\mathcal{F})&= m \cdot H^{3}\notag   
\end{align}
The formulas for the Chern characters and the Euler characteristic are obtained by applying standard computations, respectively, as in \cite[\S 14.2.1, \S 14.2.2]{EH} and \cite[\S 14.3]{EH}
 
%

The Grothendieck-Riemann-Roch theorem, for $\mathcal{F}(p,q):=\mathcal{F}\otimes\mathcal{O}(p,q),$ gives:

\begin{align}\label{euler}
\chi(\mathcal{F}(p,q))=r\cdot&\chi(\mathcal{O}(p,q))+\chi(\mathcal{O}(a,b))-1+\frac{1}{6}[3m-3k(a+4)-3l(b-2)]  \\ +&\frac{1}{2}[ap(p+a+4)+bq(q+b-2)-2(kp+lq)]. \notag
\end{align}

Grothendieck-Serre duality can be expressed as$^{\dagger}$ 
\begin{align}
\ext^{i}(\mathcal{F},\mathcal{G})\cong&\ext^{3-i}(\mathcal{G},\mathcal{F}\otimes\mathcal{O}(K_{\bp3}))^{\ast} \notag \\
\cong&\ext^{3-i}(\mathcal{G},\mathcal{F}(-4,2))^{\ast}
\end{align}

\footnotetext[2]{We are adopting the following notation; the dual of a sheaf (or a complex of sheaves, in general) $\mathcal{F},$  will be denoted by $\mathcal{F}^{\vee}.$ For $\mathbb{C}$-vector spaces ${\rm H}(\mathcal{F}),$ (or cohomology modules, in general) depending on a sheaf, or a complex of sheaves, in general, their dual will be denoted by  ${\rm H}(\mathcal{F})^{\ast}.$ }

%
%

We also remind the reader that $\bp3$ is a Fano threefold of index $i=2,$ degree $d=7$ and Picard number $2.$ 
Its fundamental divisor is $h:=\frac{-K_{\bp3}}{i}=2H-E,$ and satisfies \cite[Proposition 1.9]{Isk1}: $${\rm dim}\ho^{i}(\bp3,\mathcal{O}(h))=\left\{\begin{array}{cc} 9, & i=0  \\ 0, & i>0 \end{array}\right.$$


Moreover, one has $${\rm dim}\ho^{i}(\bp3,\mathcal{O}(nh))=0 \quad\forall i>0, \forall n>-2,$$ 
 
and in particular $\ho^{i}(\bp3,\mathcal{O}(-h))=0, \quad \forall i\geq0,$ by Serre duality.

\vspace{0.3cm}

\begin{proposition}\label{van1} \hspace{8cm}
\begin{itemize}
    \item[(i)] $\dimg\mathcal{O}(0,q)=\dimg\mathcal{O}(qE)=\left\{\begin{array}{cc}\mathcal{I}^{-q}_{p_{0}} & q<0 \\ & \\ \mathcal{O} & q\geq0\end{array}\right.,$ where $\mathcal{I}_{p_{0}}$ is the ideal sheaf of the blown-up point $p_{0}\in\p3;$
    \item[(ii)] $\ho^{0}(\mathcal{O}(p,q))=0,$ $\forall q$ if $p<0;$
    \item[(iii)] $\ho^{i}(\mathcal{O}(p,q))=0,$ $\forall i\geq0$ if $-3\leq p\leq-1,$ and $0\leq q\leq 2$ 
\end{itemize}
\end{proposition}

\vspace{0.3cm}

\begin{proof}
\begin{itemize}
    \item[(i)] By Zariski's main Theorem, one has $\dimg(\mathcal{O}_{\bp3})=\mathcal{O}_{\p3},$ $\dimg(\mathcal{O}_{E})=\mathcal{O}_{p_{0}}$ and $\mathcal{R}^{i}\dimg(\mathcal{O}_{\bp3})=\mathcal{R}^{i}\dimg(\mathcal{O}_{E})=0,$ for $i>0.$ By taking the direct images of
    \begin{align}
    0\to\mathcal{O}_{\bp3}(-E)\to\mathcal{O}_{\bp3}\to\mathcal{O}_{|E}\to0    \\
    0\to\mathcal{O}_{\bp3}\to\mathcal{O}_{\bp3}(E)\to\mathcal{O}_{|E}(-1)\to0 
    \end{align}
    One obtains
    \begin{align}
    0\to\dimg\mathcal{O}_{\bp3}(-E)\to\underbrace{\dimg\mathcal{O}_{\bp3}}_{\mathcal{O}_{\p3}}\to\underbrace{\dimg\mathcal{O}_{|E}}_{\mathcal{O}_{p_{0}}}\to\mathcal{R}^{1}\dimg\mathcal{O}_{\bp3}(-E)\to0 \label{dimg1}   \\ \notag \\
    0\to\underbrace{\dimg\mathcal{O}_{\bp3}}_{\mathcal{O}_{\p3}}\to\dimg\mathcal{O}_{\bp3}(E)\to\dimg\mathcal{O}_{|E}(-1)\to0  \label{dimg2}
    \end{align}
    
    It follows that $$\dimg\mathcal{O}_{\bp3}(-E)\cong\mathcal{I}_{p_{0}},\qquad \mathcal{R}^{1}\dimg\mathcal{O}_{\bp3}(-E)=0$$ and $$\dimg\mathcal{O}_{\bp3}(E)\cong\mathcal{O}_{\p3},\qquad \dimg\mathcal{O}_{E}(-1)=0.$$ 
    
    The rest of the result is by induction steps
    $$0\to\dimg\mathcal{O}_{\bp3}((q-1)E)\to\dimg\mathcal{O}_{\bp3}(qE)\to\dimg\mathcal{O}_{E}(-q)\to\mathcal{R}^{1}\dimg\mathcal{O}_{\bp3}((q-1)E)\to\cdots$$
    on \eqref{dimg1}, when $q<0,$ and \eqref{dimg2}, when $q\geq0.$
    

    $\ho^{0}(\mathcal{O}(p,q))=0,$ $\forall q$ if $p<0.$
    
\vspace{0.3cm}

    \item[(ii)] This follows from item ${\rm (i)}$ and the fact that \begin{align}
        \ho^{0}(\bp3,\mathcal{O}(p,q))&=\ho^{0}(\bp3,\pull\mathcal{O}_{\p3}(p)\otimes\mathcal{O}_{\bp3}(qE)) \notag \\
        &=\ho^{0}(\p3,\dimg(\pull\mathcal{O}_{\p3}(p)\otimes\mathcal{O}_{\bp3}(qE))) \notag \\
        &=\ho^{0}(\p3,\mathcal{O}_{\p3}(p)\otimes\mathcal{O}_{\bp3}(qE)); \notag
    \end{align}
    where the last line is given by the projective formula. Sine we have
    $$\ho^{0}(\p3,\mathcal{O}_{\p3}(p)\otimes\dimg\mathcal{O}_{\bp3}(q))=\left\{\begin{array}{cc}\ho^{0}(\p3,\mathcal{I}^{-q}_{p_{0}}(p))=0 &q<0  \\ \ho^{0}(\p3,\mathcal{O}_{\p3}(p))=0& q\geq0\end{array}\right.\quad \forall p<0.$$
    
\vspace{0.3cm}

    \item[(iii)] First, we note that, by item ${\rm(ii)}$ and Serre duality we have: $$\ho^{0}(\mathcal{O}(p,q))=\ho^{3}(\mathcal{O}(p,q))=0,$$ for $$(-3,0),(-3,0),(-3,0),(-2,1),(-2,1),(-2,1),(-1,2),(-1,2),(-1,2).$$  Since $\chi(\mathcal{O}(p,q))=0,$ for the above values, thereby ${\rm dim}\ho^{1}(\mathcal{O}(p,q))={\rm dim}\ho^{2}(\mathcal{O}(p,q)).$
   
Moreover, using the fact that $\ho^{i}(\mathcal{O}_{|E}(qE))\cong\ho^{2-i}(\p2, \mathcal{O}_{\p2}(q-3))^{\ast},$ by Serre duality on the plane, and the Euler Characteristic $\chi(\mathcal{O}_{|E}(qE))=\chi_{\p2}(\mathcal{O}_{\p2}(q-3))$ one has $\ho^{i}(\mathcal{O}_{|E}(1))=\ho^{i}(\mathcal{O}_{|E}(-2))=0$ for $i=0,1,2,3.$ Finally, by considering the restriction sequences $$ 0\to\mathcal{O}(p,0)\to\mathcal{O}(p,1)\to\mathcal{O}_{|E}(1)\to0$$ and $$0\to\mathcal{O}(p,1)\to\mathcal{O}(p,2)\to\mathcal{O}_{|E}(2)\to0,$$ for $p=-1,-2,-3,$ and the associated long exact sequences in cohomology, it follows that $\ho^{1}(\mathcal{O}(p,1))\cong\ho^{2}(\mathcal{O}(p,2))=0,$ for $p=-1,-2,-3.$
    
\end{itemize}
\end{proof}

For the integer coefficients appearing in the Chern Formulae \eqref{chern-ch}

\begin{lemma}\cite[lemma 4.2]{Liu}\label{van2}
For any $p,q\in\mathbb{Z},$ one has $$h^{1}(\mathcal{O}_{\bp3}(p,q))\cdot h^{2}(\mathcal{O}_{\bp3}(p,q))=0.$$
\end{lemma}

Here we put $h^{i}(\mathcal{F}):=\dim\ho^{i}(\mathcal{F})$ for a sheaf $\mathcal{F}$ on $\bp3.$

\vspace{0.3cm}

\begin{lemma}\label{intg-condition} \hspace{8cm}
\begin{itemize}
    \item[(i)] For any coherent sheaf $\mathcal{F}$ on $\bp3,$ we have $$m\equiv k(a+4)-l(b-6)\quad (mod 2)$$
    \item[(ii)] A rank $2$ reflexive sheaf on $\bp3$ is locally free if, and only if, $m=0.$ 
\end{itemize}
\end{lemma}

\vspace{0.3cm}

\begin{proof}
\begin{itemize}
    \item[(i)] Immediate, since the Euler characteristic $\chi(F)$ must be integer.
    \item[(ii)] The proof is similar to \cite[Proposition 2.6]{Hart1};
    The total Chern class of $\mathcal{F}(p,q)$ is given by \begin{align}{\rm c}(\mathcal{F}(p,q))=&1+[{\rm c}_{1}(\mathcal{F})+2{\rm c}_{1}(\mathcal{O}(p,q))] \notag \\
    +&[{\rm c}_{2}(\mathcal{F})+{\rm c}_{1}(\mathcal{F}){\rm c}_{1}(\mathcal{O}(p,q))+{\rm c}_{1}(\mathcal{O}(p,q))^{2}]+[{\rm c}_{3}(\mathcal{F})] \notag \\
    =&1+[(a+2p)H+(b+2q)E]+[(k+ap+a^{2})H^{2}+(l+qb+b^{2})E^{2}] \notag \\
    &+[mH^{3}]. \notag
    \end{align} 
    On the other hand, since $\bp3$ is smooth, hence integral and locally factorial, then for an rank $2$ coherent sheaf, one has $\mathcal{F}^{\vee}\cong\mathcal{F}\otimes\Det(\mathcal{F})^{-1}.$ Hence, by putting $p=-a$ and $q=-b,$ it follows that \begin{equation}\label{chern-dual}{\rm c}(\mathcal{F}^{\vee})=1-[aH+bE]+[kH^{2}+lE^{2}]+[mH^{3}].\end{equation}
    On the other hand, since $\mathcal{F}$ reflexive, then its cohomological dimension is less than, or equal to, $1,$ and it is given by a locally free resolution $0\to\mathcal{E}_{1}\to\mathcal{E}_{0}\to\mathcal{F}\to0.$ Applying $\mathcal{H}om(-,\mathcal{O}_{\bp3}),$ one gets: 
    $$0\to\mathcal{F}^{\vee}\to\mathcal{E}_{0}^{\vee}\to\mathcal{E}_{1}^{\vee}\to\mathcal{E}xt^{1}(\mathcal{F},\mathcal{O}_{\bp3}),$$ where the sheaf $\mathcal{E}xt^{1}(\mathcal{F},\mathcal{O}_{\bp3})$ is supported on the singularity locus of $\mathcal{F},$ hence, by reflexivity, on a finite set of points, and its Chern class is given by ${\rm c}(\mathcal{E}xt^{1}(\mathcal{F},\mathcal{O}_{\bp3}))=1+2tH^{3},$ where $t={\dim}\ho^{0}(\mathcal{E}xt^{1}(\mathcal{F},\mathcal{O}_{\bp3}))$ is its length \cite[Lemma 2.7]{Hart1}. In this way, one has ${\rm c}(\mathcal{F})={\rm c}(\mathcal{E}_{0}){\rm c}(\mathcal{E}_{1})^{-1},$ and $${\rm c}(\mathcal{F}^{\vee})={\rm c}(\mathcal{E}_{0}^{\vee})^{-1}{\rm c}(\mathcal{E}_{1}){\rm c}(\mathcal{E}xt^{1}(\mathcal{F},\mathcal{O}_{\bp3})),$$ Using the fact that for a locally free sheaf one has ${\rm c}_{i}(\mathcal{E}^{\vee})=(-)^{i}{\rm c}_{i}(\mathcal{E}^{\vee}),$ we get \begin{align}\label{chern-resol}
        {\rm c}(\mathcal{F}^{\vee})&=(1-[aH+bE]+[kH^{2}+lE^{2}]-[mH^{3}])(1+2tH^{3}) \notag \\
        &=1-[aH+bE]+[kH^{2}+lE^{2}]+[(2t-m)H^{3}]
    \end{align}
    It follows, from the above line and \eqref{chern-dual}, that $(2t-m)=m,$ hence $m=t.$ In particular, for $\mathcal{E}xt^{1}(\mathcal{F},\mathcal{O}_{\bp3})=0$ if, and only if, $\mathcal{F}$ locally free, if, and only if, $m=0.$  
\end{itemize}

\end{proof}

\vspace{0.3cm}

Another result, due to Hartshorne \cite[Theorem 2.5]{Hart1}, that can be generalized to any locally integral scheme of dimension $3,$ is the following:
\begin{theorem}\label{Serre-duality}
Let $\mathcal{F}$ be a rank $2$ reflexive sheaf on $\bp3.$ Then there are isomorphisms $$\ho^{0}(\mathcal{F}^{\vee}\otimes\omega_{\bp3})\cong \ho^{3}(\mathcal{F})^{\ast}$$ $$\ho^{3}(\mathcal{F}^{\vee}\otimes\omega_{\bp3})\cong \ho^{0}(\mathcal{F})^{\ast}$$ and an exact sequence:
$$0\to\ho^{1}(\mathcal{F}^{\vee}\otimes\omega_{\bp3}) \to \ho^{2}(\mathcal{F})^{\ast}\to \ho^{0}(\mathcal{E}xt^{1}(\mathcal{F},\omega_{\bp3}))\to \ho^{2}(\mathcal{F}^{\vee}\otimes\omega_{\bp3}) \to \ho^{1}(\mathcal{F})^{\ast}\to 0$$
\end{theorem}


The following consequence, of the above Theorem will be useful
in the next section.
\begin{theorem}\label{almost-roggero}
Let $\mathcal{F}$ be a rank $2$ reflexive sheaf on $\bp3$ with $c_{1}(\mathcal{F})=0,$ and satisfying $\ho^{2}(\mathcal{F}(-2,1))=0=\ho^{1}(\mathcal{F}(-2,1)).$ Then $\mathcal{F}$ is locally free.
\end{theorem}
\begin{proof}
From the $5$-term sequence in Theorem \ref{Serre-duality} applied to $\mathcal{F}(-2,1)$ one has  
$$0\to\ho^{1}(\mathcal{F}(-2,1)) \to \ho^{2}(\mathcal{F}(-2,1))^{\ast}\to \ho^{0}(\mathcal{E}xt^{1}(\mathcal{F}(-2,1),\omega_{\bp3}))$$ $$\qquad\qquad\qquad \to \ho^{2}(\mathcal{F}(-2,1)) \to \ho^{1}(\mathcal{F}(-2,1))^{\ast}\to 0.$$ By the vanishing conditions, we have that $\ho^{0}(\mathcal{E}xt^{1}(\mathcal{F}(-2,1),\omega_{\bp3}))=0,$ and the result follow.
\end{proof}

\section{Instanton sheaves:}\label{instanton}

Casnati-Coskun-Genk-Malaspina in \cite[Definition 1.1]{ccgm}, had a general definition for instanton bundles on Fano threefolds. We shall adopt the following definition in order to include torsion-free sheaves:

\begin{definition}\label{Instanton-def}
A rank $r$ torsion-free sheaf $\mathcal{F},$ on $\bp3,$ will be called an instanton sheaf, if it is $\mu_{L}$-semi-stable, for the polarization $L=\mathcal{O}(2,-1),$ has first Chern class $c_{1}(\mathcal{F})=0$ and satisfies:
\begin{itemize}
    \item[(i)] $\ho^{0}(\mathcal{F})=\ho^{3}(\mathcal{F}(-4,2))=0;$ 
    \item[(ii)] $\ho^{1}(\mathcal{F}(-2,1))=\ho^{2}(\mathcal{F}(-2,1))=0;$
    \item[(iii)] $\ho^{2}(\mathcal{F}(0,-1))=\ho^{2}(\mathcal{F}(-1,1))=0.$
\end{itemize}
\end{definition}

\begin{remark}\label{rmk-on-def}

In the case of a rank $2$ locally free $\mathcal{F}$, the condition on $\ho^{2}(\mathcal{F}(-2,1))$ follows from $\ho^{1}(\mathcal{F}(-2,1))=0,$ by using Serre duality. In the same way, the vanishing of $\ho^{3}(\mathcal{F}(-4,2))$ follows from $\ho^{0}(\mathcal{F})=0.$ Furthermore, in this case, the vanishings in item $(iii)$ can be proved as in \cite[Proposition 5.2]{ccgm}. We remark also that by applying Theorem \ref{almost-roggero}, it follows that there are no strictly reflexive instanton sheaves of rank $2$ on $\bp3.$ 

\end{remark}

We will also use the following
\begin{definition}\label{rk-0}
A pure $1$-dimensional sheaf $\mathcal{Q}$ on $\bp3$ is called a rank $0$- instanton sheaf if it satisfies $\ho^{0,1}(\mathcal{Q}(-2,1))=0.$
\end{definition}

In what follows we shall consider the following exceptional sequence of line bundles:

\begin{equation}\label{ex-seq}
<\mathcal{O}_{\bp3}(-2,1),\mathcal{O}_{\bp3}(-1,0),\mathcal{O}_{\bp3}(0,-1),\mathcal{O}_{\bp3}(-2,2), \mathcal{O}_{\bp3}(-1,1),\mathcal{O}_{\bp3}>
\end{equation}

\vspace{0.5cm}

\begin{lemma}\label{vanishing}
If $\mathcal{F}$ is a torsion-free instanton sheaf according to Definition \ref{Instanton-def}, then for each member $\mathcal{L}$ of the sequence \eqref{ex-seq} we have:
\begin{itemize}
    \item[(i)] $\chi(\mathcal{F}(-2,1))=0$
    \item[(ii)] $\ho^{0}(\mathcal{F}\otimes\mathcal{L})=\ho^{3}(\mathcal{F}\otimes\mathcal{L})=0;$ 
    \item[(iii)] $\ho^{2}(\mathcal{F})=\ho^{2}(\mathcal{F}(-1,0))=0.$  
\end{itemize}

\begin{proof}
\begin{itemize}
    \item[(i)] This follows from formula \eqref{euler}
    \item[(ii)] This can be seen from the vanishing $\ho^{0}(\mathcal{F})=0$ and the fact that for each $\mathcal{L},$ in the given list, we have an injection $\mathcal{F}\otimes\mathcal{L}\hookrightarrow\mathcal{F}.$ 
    On the other hand, from Definition \ref{Instanton-def}, items $(i)$ and $(ii),$ it follows that $\ho^{3}(-2,1)=0,$ moreover, one has a surjection $\ho^{3}(\mathcal{F}(-2,1))\twoheadrightarrow\ho^{3}(\mathcal{F}\otimes\mathcal{L}).$ thus the result follows.
    \item[(iii)] We start by proving that $\ho^{2}(\mathcal{F})$ consider the restriction sequence
    $$0\to\mathcal{O}_{\bp3}(-2,1)\to\mathcal{O}_{\bp3}\to\mathcal{O}_{S}\to0$$
    to the Fano surface $S.$ Twisting by $\mathcal{F}$ and taking cohomology, it follows that $\ho^{2}(\mathcal{F})\cong\ho^{2}(\mathcal{F}|_{S}).$ By Serre-Grothendieck duality on $S,$ one obtains the isomorphism $\ho^{2}(\mathcal{F})\cong\Hom(F|_{S},\omega_{S}).$ Moreover we have an injection $\Hom(F|_{S},\omega_{S})\hookrightarrow\Hom(F|_{S},\mathcal{O}_{S})\cong\ho^{2}(\mathcal{F}|_{S}\otimes\omega_{S}).$ But $\omega_{S}=\mathcal{O}|_{S}(-2,1)$ by the adjunction formula. Now twisting the restriction sequence to $S$ by $\mathcal{F}(-2,1)$ and taking cohomology, one gets   $$0=\ho^{2}(\mathcal{F}(-2,1))\to\ho^{2}(\mathcal{F}|_{S}\otimes\omega_{S})\to\ho^{3}(\mathcal{F}(-4,2))=0,$$ thus, the vanishing holds.
    To prove the second vanishing, we first remind the reader that $\bp3$ is the projective bundle $\mathbb{P}(\mathcal{O}_{\p2}\oplus\mathcal{O}_{\p2}(1))\stackrel{pr}{\to}\p2,$ over $\p2.$ We will use the pull-back of the cotangent bundle $\Omega^{1}:=pr^{\ast}\Omega^{1}_{\p2}.$ This sheaf fits in the following exact sequences:
    \begin{align}
        0\to\Omega^{1}\to\mathcal{O}_{\bp3}(-1,1)^{\oplus 3}\to\mathcal{O}_{\bp3}\to0; \label{omega1}\\
        0\to\mathcal{O}_{\bp3}(-3,3)\to\mathcal{O}_{\bp3}(-2,2)^{\oplus 3}\to\Omega^{1}\to0. \label{omega2}
    \end{align}
    The first sequence is the pull-back of the Euler sequence, and the second one is obtained by dualizing the first and twisting by $\mathcal{O}_{\bp3}(-3,3).$ 
    Twisting \eqref{omega1} by $\mathcal{F}(0,-1)$ and taking the long sequence in cohomology, one obtains the following:
    $\ho^{0}(\Omega^{1}\otimes\mathcal{F}(0,-1))=0, \ho^{3}(\Omega^{1}\otimes\mathcal{F}(0,-1))=0$ and a surjection $\ho^{2}(\Omega^{1}\otimes\mathcal{F}(0,-1))\twoheadrightarrow\ho^{0}(\mathcal{F}(-1,0))^{\oplus 3}.$ On the other hand if we twist \eqref{omega2} by $\mathcal{F}(0,-1)$ we get $$0=\ho^{2}(\mathcal{F}(-2,1))^{\oplus 3}\to\ho^{2}(\Omega^{1}\otimes\mathcal{F}(0,-1))\to\ho^{3}(\mathcal{F}(-3,2))=0,$$ and the result follows. 
\end{itemize}

\end{proof}

\end{lemma}

In \cite[Theorem 8]{AO} Ancona and Ottaviani, proved the following 
\begin{theorem}\label{AO}
Let $\mathcal{E}:=\bigoplus_{i} \mathcal{O}_{\pn}(a_{i})$ on $\pn.$ with $a_{i}\geq0,$ and for $X=\mathbb{P}(\mathcal{E})\stackrel{pr}{\to}\pn,$ let $\mathcal{U}_{rel},$ $Q$ be, respectively, the relative universal and quotient sheaf. Then any complex $\mathcal{F}^{\bullet},$ of coherent sheaves, is obtained as the cohomology of the complex $\mathcal{C}^{\bullet}_{F}$ given by 
$$\mathcal{C}^{p}_{F}=\bigoplus_{s-i=p}\bigoplus_{q+h}\mathbb{H}^{s}(pr^{\ast}\mathcal{O}_{\pn}(-q)\otimes\mathcal{U}_{rel}(-h)\otimes\mathcal{F}^{\bullet})\otimes pr^{\ast}\Omega^{p}_{\pn}(p)\otimes\Lambda^{h}Q^{\vee}. $$
\end{theorem}

Since the polarization of a projective bundle is determined up to a twist, by a line bundle, in our case, the quotient bundle is the rank $1$ bundle $\mathcal{O}_{\bp3}(1,0)$ and the relative universal bundle is the line bundle $\mathcal{O}_{\bp3}(0,1).$ As in \cite{ccgm}, this is du to the difference of the adopted notation and it can be obtained directly from the tautological sequences, \cite[Ch. 9 \S1]{EH},

$$0\to T_{X/Y}\to T\bp3\to pr^{\ast} T\p2\to0$$
in which $T_{X/Y}=\mathcal{H}om(\mathcal{U}_{rel}, Q)$ 

$$0\to\mathcal{U}_{rel}\to pr^{\ast}(\mathcal{O}_{\p2}\oplus\mathcal{O}_{\p2}(1))\to Q\to0.$$

\vspace{0.5cm}

\begin{corollary}
An instanton sheaf $\mathcal{F}$  as in definition \ref{Instanton-def}, is the middle 
cohomology of a monad $\mathcal{M}^{\bullet}$ given by   

\begin{equation}\label{ccgm-like-monad}
0\to\begin{array}{c} \mathcal{O}_{\bp3}(-1,1)^{\oplus (l+\gamma)} \\ \oplus \\ \Omega^{1}(0,-1)^{\oplus (k-l)} \end{array} \stackrel{\alpha}{\to} 
\begin{array}{c} \mathcal{O}_{\bp3}(-1,1)^{\oplus\gamma } \\ \oplus \\ \Omega^{1} (1,-1)^{\oplus k} \\ \oplus \\ \mathcal{O}_{\bp3}(-1,0)^{\oplus 2(k-l)} \end{array} 
\stackrel{\beta}{\to}\mathcal{O}_{\bp3}^{\oplus (2k-l-r)}\to0; 
\end{equation}
where $2k-l\geq r,$ $k-l\geq0$ and $\gamma:=\ho^{2}(\mathcal{F}(-2,2)).$ 
Moreover, the map $\alpha$ is injective as a sheaf map and $$Supp(Sing(\mathcal{F}))=\{x\in\bp3|\alpha(x)\textnormal{ not injective }\}.$$
\end{corollary}

We recall that the display of such a monad is the following diagram
\begin{equation}\label{display}
\xymatrix@R-1pc{
& 0\ar[d]&0\ar[d] & &  \\
& \mathcal{M}^{-1}\ar[d]\ar@{=}[r]&\mathcal{M}^{-1}\ar[d]^{\alpha} & &  \\
0\ar[r]& K \ar[r]\ar[d] &\mathcal{M}^{0}\ar[r]^{\beta}\ar[d] &\ar[r]\ar@{=}[d] \mathcal{M}^{1}&0  \\
0\ar[r]&\mathcal{F}\ar[r]\ar[d] &C\ar[r]\ar[d] &\mathcal{M}^{1}  \ar[r] &0  \\
& 0&0 & &  \\
}
\end{equation}

where we use the notation $$\mathcal{M}^{-1}:=\begin{array}{c} \mathcal{O}_{\bp3}(-1,1)^{\oplus (l+\gamma)} \\ \oplus \\ \Omega^{1}(0,-1)^{\oplus (k-l)} \end{array},\quad
\mathcal{M}^{0}:=\begin{array}{c} \mathcal{O}_{\bp3}(-1,1)^{\oplus\gamma } \\ \oplus \\ \Omega^{1} (1,-1)^{\oplus k} \\ \oplus \\ \mathcal{O}_{\bp3}(-1,0)^{\oplus 2(k-l)} \end{array}, \quad \mathcal{M}^{1}:=\mathcal{O}_{\bp3}^{\oplus (2k-l-r)};$$ $K:=\ker(\beta)$ and $C:=\coker(\alpha).$ 

\begin{proof}
The proof is a direct application of Theorem \ref{AO} by using the vanishings in lemma \ref{vanishing}.
Recall that $Supp(Sing(\mathcal{F}))=\bigcup_{1\leq p\leq3} Supp(\mathcal{E}xt^{p}(\mathcal{F},\mathcal{O}_{\bp3})).$ Dualizing the left column of the display \ref{display}, one has the exact sequence
$$0\to\mathcal{F}^{\vee}\to K^{\vee}\stackrel{\alpha^{\vee}}{\to} (\mathcal{M}^{-1})^{\vee} \to \mathcal{E}xt^{1}(\mathcal{F},\mathcal{O}_{\bp3})\to0.$$ and $\mathcal{E}xt^{2,3}(\mathcal{F},\mathcal{O}_{\bp3})=0.$ Thus, in our case we have $$Supp(Sing(\mathcal{F}))=Supp(\mathcal{E}xt^{1}(\mathcal{F},\mathcal{O}_{\bp3})),$$ Moreover, from the exact sequence above, we see that $x\in Supp(Sing(\mathcal{F}))$ if, and only if, at the level of stalks, the map $K^{\vee}_{x}\stackrel{\alpha^{\vee}_{x}}{\to} (\mathcal{M}^{-1})^{\vee}_{x}$ fails to be surjective. Hence, this occurs if, and only if, the map at the level of fibers $K^{\vee}(x)\stackrel{\alpha^{\vee}(x)}{\to} (\mathcal{M}^{-1})^{\vee}(x)$ is not surjective, since both $K^{\vee}$ and $(\mathcal{M}^{-1})^{\vee}$ are locally free. In other words, $Supp(Sing(\mathcal{F}))$ is the locus where $\alpha$ fails to be injective.
\end{proof}


The main result of this section is a generalization of \cite[Main Theorem]{JG} and \cite[Theorem 3.4]{henni} in the case of the torically fixed rank $2$ instantons. We shall prove the following:

\begin{theorem}\label{main1}
If $\mathcal{F}$ is a strictly torsion-free rank $2$ instanton sheaf , then
\begin{itemize}
    \item[(i)]$\mathcal{F}^{\vee}\cong\mathcal{F}^{\vee\vee}$ and both are locally free; 
    \item[(ii)] The singularity locus of $\mathcal{F}$ is of pure dimension $1;$
    \item[(iii)] The sheaves $\mathcal{E}xt^{1}(\mathcal{F},\mathcal{O}_{\bp3})$ and $\mathcal{F}^{\vee\vee}/\mathcal{F}$ are both rank-zero instanton sheaves of charge $c_{2}(\mathcal{F})-c_{2}(\mathcal{F}^{\vee\vee})$ and satisfy $$\mathcal{E}xt^{2}(\mathcal{E}xt^{1}(\mathcal{F},\mathcal{O}_{\bp3}),\omega_{\bp3})\cong (\mathcal{F}^{\vee\vee}/\mathcal{F})\otimes\omega_{\bp3}.$$
\end{itemize}

\end{theorem}

Item $(i)$ is the content of the last part of Remark \ref{rmk-on-def}, following from Theorem \ref{almost-roggero}, since the dual, and hence the double dual, of any sheaf is reflexive.  The proof for items $(ii)$ will be given at the end of this section, after showing some intermediate results. 

\begin{lemma}\label{v-vanish}
If $\mathcal{F}$ is a torsion free instanton sheaf on $\bp3$ then
\begin{itemize}
    \item[(i)] $\ho^{1}(\mathcal{F}^{\vee}(-2,1))=\ho^{0}(\mathcal{F}^{\vee})=0,$ 
    \item[(ii)] $\ho^{0}(\mathcal{E}xt^{1}( \mathcal{F},\mathcal{O}_{\bp3}(-2,1)))=\ho^{1}(\mathcal{E}xt^{1}( \mathcal{F},\mathcal{O}_{\bp3}(-2,1)))=0,$ and \\  $\ho^{2}(\mathcal{F}^{\vee}(-2,1))=0,$
    \item[(iii)]$\ho^{0,1}((\mathcal{F}^{\vee\vee}/\mathcal{F})(-2,1))=0.$
\end{itemize}
\end{lemma}

Remark that from this lemma one can deduce that the dual of a rank $2$ torsion free instanton sheaf is an instanton bundle according to the definition from \cite{ccgm}.

\begin{proof}
Dualizing the monad \eqref{ccgm-like-monad} one gets a complex 
\begin{equation}\label{dual-complex}
\widetilde{\mathcal{M}}^{\bullet}:\quad 0\to(\mathcal{M}^{1})^{\vee}\to(\mathcal{M}^{0})^{\vee}\to(\mathcal{M}^{-1})^{\vee}\to0 \end{equation} whose cohomology is given by $\mathcal{H}^{0}(\widetilde{\mathcal{M}})=\mathcal{F}^{\vee}$ and $\mathcal{H}^{1}(\widetilde{\mathcal{M}})=\mathcal{E}xt^{1}(\mathcal{F},\mathcal{O}_{\bp3}).$ The display associated to it can be written as:

\begin{equation}\label{d-display}
\xymatrix@R-1pc{
& 0\ar[d]&0\ar[d] & &  \\
& (\mathcal{M}^{1})^{\vee}\ar[d]\ar@{=}[r]&(\mathcal{M}^{1})^{\vee}\ar[d]^{\alpha} & &  \\
0\ar[r]& C^{\vee} \ar[r]\ar[d] &\mathcal{M}^{0}\ar[r]^{\beta}\ar[d] &\ar[r]\ar@{=}[d] \mathcal{G}&0  \\
0\ar[r]&\mathcal{F}^{\vee}\ar[r]\ar[d] &K^{\vee}\ar[r]\ar[d] &\mathcal{G}  \ar[r] &0  \\
& 0&0 & &  \\
}
\end{equation}
where $\mathcal{G}$ is a sheaf that fits in the following short exact sequence: 
\begin{equation}\label{G-sheaf}
0\to\mathcal{G}\to(\mathcal{M}^{-1})^{\vee}\to\mathcal{E}xt^{1}(\mathcal{F},\mathcal{O}_{\bp3})\to0,
\end{equation}
and whose double dual is $\mathcal{M}^{-1}.$

\begin{itemize}
    \item[(i)] By twisting the left column by $\mathcal{O}_{\bp3}(-2,1))$ and taking the long exact sequence in cohomology one gets $$\ho^{1}(C^{\vee}(-2,1))\to\ho^{1}(\mathcal{F}^{\vee}(-2,1))\to \ho^{2}((\mathcal{M}^{1})^{\vee}(-2,1)),$$ in which the term on the right is zero, since $$\ho^{2}((\mathcal{M}^{1})^{\vee}(-2,1))=\ho^{2}(\mathcal{O}_{\bp3}(-2,1))^{2k-l-2}=0.$$ Using the sequence middle row of \eqref{d-display}, the term on the left can be put in the sequence $$\ho^{0}(\mathcal{G}(-2,1))\to \ho^{1}(C^{\vee}(-2,1))\to\ho^{1}((\mathcal{M}^{0})^{\vee}(-2,1)).$$ The term on the right is zero by using the exact sequence \eqref{omega1} and the vanishings in item ${\rm (iii)}$ in Proposition \ref{van1}. Furthermore, $\ho^{0}(\mathcal{G}(-2,1))=0,$ since $\ho^{0}(\mathcal{G}(-2,1))\hookrightarrow\ho^{0}((\mathcal{M}^{-1})^{\vee}(-2,1))=\ho^{0}(\mathcal{O}_{\bp3}(-1,0))^{\oplus(l+\gamma)}\oplus \ho^{0}(\Omega^{1}(1,-1)).$ The first summand is cohomologically trivial and from \eqref{omega2} one has  \begin{equation}\label{middle1}\ho^{0}(\mathcal{O}_{\bp3}(-1,1)\to \ho^{0}(\Omega^{1}(1,-1))\to\ho^{1}(\mathcal{O}_{\bp3}(-2,2)).\end{equation} Both left and right terms vanish, since also $\mathcal{O}_{\bp3}(-2,2)$ and $\mathcal{O}_{\bp3}(-1,1)$ are cohomologically trivial. 
    
    To prove that $\mathcal{F}^{\vee}$ has no sections, recall that from Remark \ref{rmk-on-def} that this sheaf is locally free and since it is of rank $2$ and has trivial Chern class, we have $\mathcal{F}^{\vee}\cong\mathcal{F}^{\vee\vee}.$ Dualizing again the complex \eqref{dual-complex}, we obtain a monad whose terms are equal to the terms in $\mathcal{M}^{\bullet},$ and whose maps, $\alpha$ and $\beta$ are different, in particular its first map $\alpha^{\vee\vee}(x)$ is injective everywhere. Hence, one can fit $\ho^{0}(\mathcal{F}^{\vee})$ in the sequence \begin{equation}\label{middle2} \ho^{0}(K)\to\ho^{0}(\mathcal{F}^{\vee})\to\underbrace{\ho^{1}\mathcal{O}_{\bp3}(-1,1))}_{0}\oplus\ho^{1}(\Omega^{1}(0,-1)).\end{equation} From \eqref{omega2} one has  $$\ho^{1}(\mathcal{O}_{\bp3}(-2,1)\to \ho^{1}(\Omega^{1}(0,-1)\to\ho^{2}(\mathcal{O}_{\bp3}(-3,2)),$$ and it follows that the right term, in \eqref{middle2}, is trivial. Finally, we have an injection $$\ho^{0}(K)\hookrightarrow\ho^{0}(\mathcal{M}^{0})=\left\{\begin{array}{c}\ho^{0}(\mathcal{O}_{\bp3}(-1,1))^{\oplus \gamma}=0 \\ \oplus \\ \ho^{0}(\Omega^{1}(1,-1))^{\oplus k}=0 \textnormal{ by \eqref{middle1}}\\ \oplus \\ \ho^{0}(\mathcal{O}_{\bp3}(-1,0))^{\oplus 2(k-l)}=0\end{array}\right.$$   
\vspace{0.5cm}

    \item[(ii)] Twisting \eqref{G-sheaf}, by $\mathcal{O}_{\bp3}(-2,1)$ and taking cohomology, we get the exact sequence \begin{align}&\ho^{0}((\mathcal{M}^{-1})^{\vee}(-2,1))\to\ho^{0}(\mathcal{E}xt^{1}(\mathcal{F},\mathcal{O}_{\bp3}(-2,1)))\to\ho^{1}(\mathcal{G}(-2,1)) \notag \\ &\to\ho^{1}((\mathcal{M}^{-1})^{\vee}(-2,1)).\notag \end{align}  

Recall that $$(\mathcal{M}^{-1})^{\vee}(-2,1)=\begin{array}{c}\Omega^{1}(1,-1)^{k-l} \\ \oplus \\ \mathcal{O}_{\bp3}(-1,0)^{l-\gamma} \end{array}$$
 and $\mathcal{O}_{\bp3}(-1,0)$ is cohomologically trivial. On the other hand, from the long exact sequence in cohomology associated to \eqref{omega2} twisted by $\mathcal{O}_{\bp3}(1,-1)$  one has 
    $$\ho^{0}(\mathcal{O}_{\bp3}(-1,1))\to \ho^{0}(\Omega^{1}(1,-1))\to \ho^{1}(\mathcal{O}_{\bp3}(-2,2))$$ and $$\ho^{1}(\mathcal{O}_{\bp3}(-1,1))\to \ho^{1}(\Omega^{1}(1,-1))\to \ho^{2}(\mathcal{O}_{\bp3}(-2,2)).$$ Since the line bundles $\mathcal{O}_{\bp3}(-2,2)$ and $\mathcal{O}_{\bp3}(-1,1)$ are both cohomologically trivial, it follows that $\ho^{0}(\mathcal{E}xt^{1}(\mathcal{F},\mathcal{O}_{\bp3}(-2,1)))\cong\ho^{1}(\mathcal{G}(-2,1)).$
    Now, if we consider the lower row of the display \eqref{d-display} twisted by $\mathcal{O}_{\bp3}(-2,1),$ and take the long exact sequence in cohomology, we obtain $$\ho^{1}(K^{\vee}(-2,1))\to\ho^{1}(\mathcal{G}(-2,1))\to\ho^{2}(\mathcal{F}^{\vee}(-2,1))\to\ho^{2}(K^{\vee}(-2,1)).$$  
    On the other hand $$\ho^{1}(\mathcal{M}^{0})^{\vee}(-2,1))\to \ho^{1}(K^{\vee}(-2,1))\to\ho^{2}(\mathcal{O}_{\bp3}(-2,1))$$ and $$\ho^{2}(\mathcal{M}^{0})^{\vee}(-2,1))\to \ho^{2}(K^{\vee}(-2,1))\to\ho^{3}(\mathcal{O}_{\bp3}(-2,1)).$$ A similar computation as above, for the case of $(\mathcal{M}^{-1})^{\vee}$ and using the fact that the line bundles involved are cohomologically trivial, one obtains that the isomorphism $\ho^{1}(\mathcal{G}(-2,1))\cong\ho^{2}(\mathcal{F}^{\vee}(-2,1))$ Thus $$\ho^{2}(\mathcal{F}^{\vee}(-2,1))\cong\ho^{0}(\mathcal{E}xt^{1}(\mathcal{F},\mathcal{O}_{\bp3}(-2,1))).$$ But $\ho^{2}(\mathcal{F}^{\vee}(-2,1))=\ho^{1}(\mathcal{F}^{\vee}(-2,1)),$ and from item $(i)$ the latest term is zero.
    The proof of the vanishing $\ho^{1}(\mathcal{E}xt^{1}( \mathcal{F},\mathcal{O}_{\bp3}(-2,1)))$ is similar.
    
    \item[(iii)] Let $\mathcal{Q}:=\mathcal{F}^{\vee\vee}/\mathcal{F}.$ From the exact sequence \begin{equation}\label{db-dual}0\to\mathcal{F}\to\mathcal{F}^{\vee\vee}\to\mathcal{Q}\to0\end{equation} twisted by $\mathcal{O}_{\bp3}(-2,1)$ one has $$0=\ho^{i}(\mathcal{F}^{\vee\vee}(-2,1))\to\ho^{i}(\mathcal{Q}(-2,1))\to\ho^{i+1}(\mathcal{F}(-2,1))=0.$$ The vanishing of the left hand term can be obtained by twisting the monad $(\mathcal{M}^{\bullet})^{\vee\vee}$ by $\mathcal{O}_{\bp3}(-2,1)$ and applying the vanishings in Proposition \ref{van1}, item ${\rm (iii)}.$

\end{itemize}
\end{proof}

\begin{lemma}\label{ext-vanish}
\begin{align}
    \mathcal{E}xt^{1,3}(\mathcal{E}xt^{1}(\mathcal{F},\mathcal{O}_{\bp3}),\omega_{\bp3})=0, \qquad and \notag  \\ 
\mathcal{E}xt^{2}(\mathcal{E}xt^{1}(\mathcal{F},\mathcal{O}_{\bp3}),\omega_{\bp3})\cong(\mathcal{F}^{\vee\vee}/\mathcal{F})\otimes\omega_{\bp3}. \notag
\end{align}
\end{lemma}

\begin{proof}
combining \eqref{G-sheaf} and the lower row in \eqref{d-display} one has a diagram
    $$\xymatrix@R-1pc{
     & & & 0\ar[d]& & \\
     0\ar[r]&\mathcal{M}^{-1}\ar[r]\ar[d] &K\ar[r]\ar@{=}[d] &\mathcal{F}\ar[r]\ar[d] &0 & \\
     0\ar[r]& \mathcal{G}^{\vee}\ar[r]&K\ar[r] & \mathcal{F}^{\vee\vee}\ar[r]\ar[d]& \mathcal{E}xt^{1}(\mathcal{G},\mathcal{O}_{\bp3})\ar[r]&0 \\
     & & & \mathcal{Q} \ar[d]& & \\
     & & &0 & & \\ 
    }$$
    and by the snake lemma it follows that $\mathcal{E}xt^{1}(\mathcal{G},\mathcal{O}_{\bp3})\cong\mathcal{Q}.$
    Moreover, by dualizing \eqref{G-sheaf} one has $\mathcal{E}xt^{1}(\mathcal{E}xt^{1}(\mathcal{F},\mathcal{O}_{\bp3}),\mathcal{O}_{\bp3})=0,$ $\mathcal{E}xt^{1}(\mathcal{G},\mathcal{O}_{\bp3})\cong\mathcal{E}xt^{2}(\mathcal{E}xt^{1}(\mathcal{F},\mathcal{O}_{\bp3}),\mathcal{O}_{\bp3})$ and $\mathcal{E}xt^{2}(\mathcal{G},\mathcal{O}_{\bp3})\cong\mathcal{E}xt^{3}(\mathcal{E}xt^{1}(\mathcal{F},\mathcal{O}_{\bp3}),\mathcal{O}_{\bp3}).$ Finally, by dualizing the middle row in \eqref{d-display} it is easy to see that $\mathcal{E}xt^{2}(\mathcal{G},\mathcal{O}_{\bp3})=0.$ A final twist by $\omega_{\bp3}$ gives the result.  

\end{proof}

\begin{proof}[Proof of Theorem \ref{main1}]

By Lemma \ref{ext-vanish} we have $$\codim(\mathcal{E}xt^{3}(\mathcal{E}xt^{1}(\mathcal{F},\mathcal{O}_{\bp3}),\omega_{\bp3}))=4,$$ since it has no support on $\bp3,$ and it follows from \cite[Proposition 1.1.10]{Huy} that $\mathcal{E}xt^{1}(\mathcal{F},\mathcal{O}_{\bp3})$ is pure of dimension $1.$ Moreover, by item $(ii)$ of Lemma \ref{v-vanish} this sheaf is a rank $0$ instanton sheaf.

On the other hand, by dualizing the sequence \eqref{db-dual}, one has $\mathcal{E}xt^{1}(\mathcal{F},\mathcal{O}_{\bp3})=\mathcal{E}xt^{2}(\mathcal{Q},\mathcal{O}_{\bp3})$ and $\mathcal{E}xt^{3}(\mathcal{Q},\mathcal{O}_{\bp3})=\mathcal{E}xt^{2}(\mathcal{F},\mathcal{O}_{\bp3})=0.$ Hance also $\mathcal{Q}$ is pure of dimension $1.$ It also follows from sequence \eqref{db-dual} that $\mathcal{Q}$ is a rank $0$ instanton sheaf, since $\ho^{0,1}(\mathcal{Q}(-2,1))=0,$ from item $(iii)$ of Lemma \ref{v-vanish}.
Finally the charge of these instanton is easily obtained.

\end{proof}

\section{Elementary transformations}\label{element-transf}

In this section we shall give examples of strictly torsion-free instanton sheaves on $\bp3;$ We start with the following:

\begin{definition}\label{elementary}
Let $\mathcal{E}$ be a rank $2$ instanton sheaf on $\bp3,$ of charge $kH^{2}+lE^{2};$ a triple $(\Sigma, \mathcal{L},\phi)$ will be called an elementary transformation data for $\mathcal{E}$ if it satisfies 
\begin{itemize}
    \item[(i)] $i:\Sigma\hookrightarrow\bp3$ is an embedded reduced, but possibly reducible, locally complete intersection curve of genus $g,$ with class $[\Sigma]=d_{1}H^{2}+d_{2}E^{2}\in A^{2}(\bp3);$ 
    \item[(ii)] $\phi:\mathcal{E}\twoheadrightarrow i_{\ast}\mathcal{L}\otimes\mathcal{O}_{\bp3}(2,-1)$ is a surjective morphism.
    \item[(iii)] $\mathcal{L}\in\pic^{g-1}(\Sigma)$ a line bundle on $\Sigma$ satisfying $\ho^{0,1}(i_{\ast}\mathcal{L})=0,$ and such that the following maps $$\ho^{1}(\mathcal{E}(-1,1))\twoheadrightarrow\ho^{1}(i_{\ast}\mathcal{L}\otimes\mathcal{O}_{\bp3}(1,0))\textnormal{ and } $$ $$\ho^{1}(\mathcal{E}(-1,0))\twoheadrightarrow\ho^{1}(i_{\ast}\mathcal{L}\otimes\mathcal{O}_{\bp3}(1,-1))$$ are surjective;

\end{itemize}
In this case, the sheaf $\mathcal{F}:=\ker(\phi)$ is called an elementary transformation of $\mathcal{E}.$ 
\end{definition}

The above definition also appears in \cite{JMT1} and \cite{MT}. Some examples will be given in {\bf Example} \ref{examples} bellow.

\begin{proposition}\label{transform-also-instanton}
Let $\mathcal{E}$ be a rank $2$ torsion-free instanton of charge $kH^{2}+lE^{2},$ Then its elementary transform $\mathcal{F}$ with respect to the data $(\Sigma,\mathcal{L},\phi)$ is also a rank $2$ instanton sheaf of charge $(k+d_{1})H^{2}+(l+d_{2})E^{2},$ and satisfies $\mathcal{F}^{\vee}\cong\mathcal{E}^{\vee}.$     
\end{proposition}

\begin{proof}
First, by taking the long exact sequence of cohomology associated to the  elementary transformation sequence 
\begin{equation}\label{transform-sequence}
    0\to\mathcal{F}\to\mathcal{E}\stackrel{\phi}{\to}i_{\ast}\mathcal{L}(2,-1)\to0
\end{equation}
twisted by $\mathcal{O}_{\bp3}(-2,1)$ and by using the vanishing $\ho^{0,1}(i_{\ast}\mathcal{L})=0,$ in item $(ii)$ of Definition \ref{elementary}, one concludes that $\ho^{1,2}(\mathcal{F}(-2,1)).$ On the other hand, twisting by the sequence above by either $\mathcal{O}_{\bp3}(0,-1)$ or $\mathcal{O}_{\bp3}(-1,1)$ and using the surjectivity of the morphisms, again, in item $(ii)$ of Definition \ref{elementary}, it follows that $\ho^{2}(\mathcal{F}(0,-1))$ and $\ho^{2}(\mathcal{F}(-1,1)).$ Finally, the vanishings of item $(i)$ in Definition \ref{Instanton-def} are easily obtained, since $i_{\ast}\mathcal{L}$ is supported in dimension $1$ and $\mathcal{F}(-1,1)$ is a subsheaf of $\mathcal{E}(-1,1).$ Moreover if $\mathcal{F}$ is not semi-stable, then a destabilizing ideal sheaf of $\mathcal{F},$ would also destabilize $\mathcal{E},$ contradicting our hypothesis.
It remains to compute the instanton charge of $\mathcal{F}.$ This is an application of the formula 
\begin{equation}\label{chern-pull-back}
    \ch(i_{\ast}\mathcal{L}\otimes\mathcal{O}_{\bp3}(2,-1))=i_{\ast}[\ch(\mathcal{L}\otimes i^{\ast}\mathcal{O}_{\bp3}(2,-1)){\rm Td}(N_{\Sigma/\bp3})^{-1}]
\end{equation}
where $N_{\Sigma/\bp3}$ is the normal sheaf to $\Sigma$ in $\bp3$ and its Todd class can be computed from the sequence $$0\to{\rm T}_{\Sigma}\to i_{\ast}{\rm T}_{\bp3}\to N_{\Sigma/\bp3}\to0,$$ 
and in our case, it is given by ${\rm Td}(N_{\Sigma/\bp3})^{-1}=1-(2d_{1}-d_{2}+g-1)\cdot [{\rm pt}],$ where $[{\rm pt}]$ is the class of a point. Since $\mathcal{L}\in\pic^{g-1}(\Sigma),$ it follows that 
\begin{align}
\ch(\mathcal{L}\otimes i^{\ast}\mathcal{O}_{\bp3}(2,-1))&=\ch((\mathcal{L}\otimes i^{\ast}\mathcal{O}_{\Sigma}((2d_{1}-d_{2})[{\rm pt}]))) \notag \\
&=1+(g-1+2d_{1}-d_{2})[{\rm pt}] \notag
\end{align}
Hence by \eqref{chern-pull-back} it follows that $\ch_{3}(\mathcal{L}\otimes i^{\ast}\mathcal{O}_{\bp3}(2,-1))=0.$
Moreover, since $c_{1}(\mathcal{L}\otimes i^{\ast}\mathcal{O}_{\bp3}(2,-1))=0,$ it follows that $\mathcal{F}$ is a rank $2$ instanton sheaf of charge $c_{2}(\mathcal{F})=(k+d_{1})H^{2}+(l+d_{2})E^{2}$

\end{proof}

\begin{remark}\label{torsion-free-existence}
If we start with a rank $2$ locally free instanton sheaf $\mathcal{E},$ then, by construction, its elementary transform $\mathcal{F},$ with respect to a datum $(\Sigma,\mathcal{L}, \phi),$ will be strictly torsion-free. Furthermore, by Proposition \ref{transform-also-instanton}, this is an instanton sheaf as in Definition \ref{Instanton-def}, with double dual $\mathcal{F}^{\vee\vee}=\mathcal{E},$ which proves their existence.
\end{remark}

\begin{example}\label{examples}
\begin{itemize}
\item[(i)] One can take $\Sigma:=L$ where $L$ is a line whose class $H^{2}$ in the Chow ring of $\bp3.$ Then by taking $\mathcal{L}:=\mathcal{O}_{L}(-1),$ it is easy to see that the conditions $\ho^{0,1}(i_{\ast}\mathcal{L})=0,$ are satisfied, and that the maps in item $(ii)$ of \ref{elementary} are surjective since $\ho^{1}(i_{\ast}\mathcal{L}\otimes\mathcal{O}_{\bp3}(1,0))=\ho^{1}(\mathcal{O}_{L})=0$ and $\ho^{1}(i_{\ast}\mathcal{L}\otimes\mathcal{O}_{\bp3}(1,-1))=\ho^{1}(\mathcal{O}_{L})=0.$ Thus the elementary transform of a rank $2$ instanton sheaf of charge $kH^{2}+lE^{2},$ is an instanton sheaf of charge $(k+1)H^{2}+lE^{2},$

\item[(ii)] One can also take $\Sigma:=L$ where, now, $L$ is a line whose class $E^{2}$ and setting $\mathcal{L}:=\mathcal{O}_{L}(-1),$ again, then $\ho^{0,1}(i_{\ast}\mathcal{L})=0,$ as above and surjectivity of the maps in item $(ii)$ of \ref{elementary} is also satisfied since $\ho^{1}(i_{\ast}\mathcal{L}\otimes\mathcal{O}_{\bp3}(1,0))=\ho^{1}(\mathcal{O}_{L}(-1))=0$ and $\ho^{1}(i_{\ast}\mathcal{L}\otimes\mathcal{O}_{\bp3}(1,-1))=\ho^{1}(\mathcal{O}_{L})=0.$ The elementary transform of the instanton sheaf of charge $kH^{2}+lE^{2},$ in this case, will be of charge $kH^{2}+(l+1)E^{2}.$

\item[(iii)] More generally if we set $\Sigma:=L_{1}\bigsqcup L_{2},$ where $L_{1}$ is a line whose class $H^{2},$ and $L_{2}$ is a line whose class $E^{2},$ then a concatenation of elementary transformations of items $(i)$ and $(ii)$ can be seen as an elementary transformation with respect to a pure $1$ dimensional sheaf $Q:=i_{1\ast}\mathcal{O}_{L_{1}}(2)\oplus i_{2\ast}\mathcal{O}_{L_{2}}(2),$ We let the reader check that all the conditions are satisfied. Moreover, this can, also, be generalized to the case where $\Sigma:=C_{1}\bigsqcup C_{2},$ where $C_{1}$ is a curve whose class $d_{1}H^{2},$ and $C_{2}$ is another curve whose class $d_{2}E^{2}.$  
\end{itemize}
\end{example}

For more details on the Hilbert scheme of line on Fano threefolds, we refer the reader to \cite{kuznetsov2}

\begin{lemma}\label{mu-stable}
If $\mathcal{E}$ is a $\mu_{L}$-stable rank $2$ instanton sheaf on $\bp3,$ then its elementary transform is also $\mu_{L}$-stable.
\end{lemma}
\begin{proof}
If the elementary transform $\mathcal{F}$ of $\mathcal{E}$ is not stable, it admits a destabilizing sheaf $\mathcal{G},$ hence with $\mu_{L}(G)=0.$ Thus it destabilizes $\mathcal{E},$ which contradicts the hypothesis.
\end{proof}

Recall that instanton bundles $\mathcal{E}$ given in \cite[Construction 7.2]{ccgm} are obtained, via Hartshorne-Serre correspondence, from a union of disjoint lines, so we shall call them \emph{t'Hooft instantons on} $\bp3,$ since they generalize the t'Hooft instantons on $\p3.$ If an instanton is not given by such a construction then we shall say that it is non t'Hooft. 

The t'Hooft instanton bundles, on $\bp3,$ of any admissible charge, are proved to be $\mu_{L}$-stable for \cite[Theorem 1.7]{ccgm}, These are {\em earnest instantons}, that is, they also satisfy the additional condition $h^{1}(\bp3,\mathcal{E}(-h-D))=0,$ for a non empty linear system $|D|$ that contains smooth integral elements \cite[Definition 1.2]{ccgm}. The question of their existence \cite[Question 1.10]{ccgm} has not been answered yet. For the more general instantons bundle, earnest or not, the claim about $\mu_{L}$-stability have been proved positive only for small values of the charge \cite[Proposition 6.4]{ccgm}. The following conclusion, broaden the result in the case of higher values of the charge and non locally free case: 

\begin{corollary}\label{stable-existence}
There exist $\mu_{L}$-stable rank $2$ instanton sheaves (earnest, or not), of any charge $kH^{2}+lE^{2}$ .
\end{corollary}

\begin{proof}
From \cite[Proposition 6.4]{ccgm}, all instanton bundles of charge $kH^{2}+lE^{2},$ such that $(k-l)\leq 14$ are stable, in particular, if we start with a non earnest instanton, then by, iterated elementary transformations and Lemma \ref{mu-stable}, we fill the gaps for $(k-l)>14.$ 
\end{proof}

If some of these non-locally free $\mu_{L}$-stable sheaves can be deformed to locally free instantons, and provided that they are non earnest, then they might be useful to prove existence of non earnest instanton bundles, in particular $\mu_{L}$-stable one


\section{An application: Smoothability}\label{application}

As an application, of the results above,  we want to study the t'Hooft component of instantons on $\bp3,$ In particular, we will show that the locus of isomorphism classes obtaind, via elementary transformations, from a t'Hooft instanton bundle, are in the boundary of a higher charge t'Hooft component. Thus they are smoothable. Furthermore we will see that such a locus is smooth, hence it cannot be in the intersection of two torsion free components of instantons. This requires the following results, proved in \cite{JMT, JMT1}, for $\p3,$ and which can be generalized easily to $\bp3;$

\begin{proposition}[Lemma 5.1 \cite{JMT1}]\label{hom-quotients}
Let $X$ be a 3-fold, $i:\Sigma\to X$ a reduced, locally complete intersection curve, $M$ a line bundle on $\Sigma$ and $\mathcal{E}$ a rank $2$ locally free sheaf on $X$ equipped with a surjective map $\phi:\mathcal{E}\twoheadrightarrow i_{\ast}M.$ Then for the the torsion free sheaf $\mathcal{F}:=\ker(\phi),$we have an exact sequence \begin{equation}\label{hom-sequence}0\to\mathcal{H}om(\mathcal{F},\mathcal{F})\to\mathcal{H}om(\mathcal{E},\mathcal{E})\to i_{\ast}M^{\otimes 2}\otimes \det(\mathcal{E})^{-1}\to0.\end{equation}
\end{proposition}

By using the result above one can generalize \cite[Lemma 5.2]{JMT1} to get the following;

\begin{lemma}\label{smooth-elementary}
Let $\mathcal{E}$ be a $\mu$-stable instanton bundle of charge $kH^{2}+lE^{2}$ on $\bp3,$ and $(\Sigma,\mathcal{L},\phi)$ be an elementary transformation data for $\mathcal{E},$ with $\ho^{\ast}(i_{\ast}\mathcal{L}^{\otimes 2}\otimes\mathcal{O}_{\bp3}(4,-2))=0,$ then The elementary transform $\mathcal{F}=\ker(\phi)$ satisfies:

\begin{itemize}
    \item [(i)]$\ext^{1}(\mathcal{F},\mathcal{F})=\ho^{0}(\mathcal{E}xt^{1}(\mathcal{F},\mathcal{F}))\oplus\ho^{1}(\mathcal{H}om(\mathcal{F},\mathcal{F}));$
    \item[(ii)]$\ext^{2}(\mathcal{F},\mathcal{F})=\ho^{1}(\mathcal{E}xt^{1}(\mathcal{F},\mathcal{F}))\cong\ho^{1}(\mathcal{E}xt^{2}(i_{\ast}\mathcal{L}\otimes\mathcal{O}_{\bp3}(2,-1),\mathcal{F}));$ 
    \item[(iii)]$h^{1}(\mathcal{H}om(\mathcal{F},\mathcal{F}))=8k-4l-3+h^{0}(i_{\ast}\mathcal{L}^{\otimes 2}\otimes\mathcal{O}_{\bp3}(4,-2)).$
\end{itemize}
\end{lemma}

\begin{proof}
As in the proof for \cite[Lemma 5.2]{JMT1}, we shall use the local-to-global spectral sequence with $E_{2}$-term $$E_{2}^{pq}=\ho^{p}(\mathcal{E}xt^{q}(\mathcal{F},\mathcal{F}))$$ given by the table:
\begin{equation}\label{table}
\begin{tabular}{|p{2.5cm}|p{2.5cm}|p{0.5cm}|p{0.5cm}|}\hline
   0  & 0 & 0 & 0\\ \hline 
    0  & 0 & 0 & 0\\ \hline 
    $\ho^{1}(\mathcal{H}om(\mathcal{F},\mathcal{F}))$ & $\ho^{1}(\mathcal{E}xt^{1}(\mathcal{F},\mathcal{F}))$ & 0 & 0\\ \hline 
    $\ho^{0}(\mathcal{H}om(\mathcal{F},\mathcal{F}))$ & $\ho^{0}(\mathcal{E}xt^{1}(\mathcal{F},\mathcal{F}))$ & 0 & 0\\ \hline 
\end{tabular}
\end{equation}

The utmost right column and the one on its left are formed from zero terms because $\mathcal{E}xt^{q}(\mathcal{F},\mathcal{F})=0$ for $q=2,3.$ This can be seen by applying $\mathcal{H}om(\bullet ,\mathcal{F})$ on the left column of display \eqref{display}. The vanishing of the terms in the third column from the right are obtained by the fact that $\mathcal{E}xt^{1}(\mathcal{F},\mathcal{F})$ is supported on the curve $\Sigma.$ 
By applying the cohomology functor on the sequence \eqref{hom-sequence} one obtains 

\begin{align}
    0&\to\ho^{0}(\mathcal{H}om(\mathcal{F},\mathcal{F}))\to\ho^{0}(\mathcal{H}om(\mathcal{E},\mathcal{E}))\to \ho^{0}( i_{\ast}\mathcal{L}^{\otimes 2}\otimes \mathcal{O}_{\bp3}(4,-2)) \label{long} \\
&\to\ho^{1}(\mathcal{H}om(\mathcal{F},\mathcal{F}))\to\ho^{1}(\mathcal{H}om(\mathcal{E},\mathcal{E}))\to 0    \notag \\
&\ho^{2}(\mathcal{H}om(\mathcal{F},\mathcal{F}))\to\ho^{2}(\mathcal{H}om(\mathcal{E},\mathcal{E}))=0 \label{ext2} \\
&\ho^{3}(\mathcal{H}om(\mathcal{F},\mathcal{F}))\to\ho^{3}(\mathcal{H}om(\mathcal{E},\mathcal{E}))=0 \label{ext3}
\end{align}

Since $\ho^{1}( i_{\ast}\mathcal{L}^{\otimes 2}\otimes\mathcal{O}_{\bp3}(4,-2))$ vanishes by hypothesis.
Furthermore, the vanishings \eqref{ext2} and \eqref{ext3} hold because $\mathcal{E}$ is a locally free instanton sheaf on $\bp3$ and the result in \cite[Theorem 1.7]{ccgm}. At this stage, we already see that the above spectral sequence degenerates at the second step. Hence item $(i)$ and the first equality in item $(ii)$ follow immediately. 
The last isomorphism in item $(ii)$ is obtained by applying the functor $\mathcal{H}om(\bullet,\mathcal{F})$ to the sequence \eqref{transform-sequence}.
Finally, item $(iii)$ follows from sequence \eqref{long}; since $\mathcal{E}$ is $\mu$-stable, then by Corollary \eqref{stable-existence}, also $\mathcal{F}$ is $\mu$-stable, hence both are simple. Hence the formula follows immediately from \cite[Theorem 1.7]{ccgm}.

\end{proof}

\begin{lemma}\label{surjective-family}\cite[Lemma 7.1]{JMT1}
Let $C$ be a smooth irreducible curve with marked point $0,$ and set $\mathbf{B}:=C\times X,$ where $X$ is a smooth threefold. Let $\mathbf{F}$ and $\mathbf{G}$ be $\mathcal{\mathbf{B}}$-sheaves, flat over $C$ and such that $\mathbf{F}$ is locally free along ${\rm Supp}(\mathbf{G}).$ Denote $X_{t}=\{t\}\times X,$ and $$F_{t}=\mathbf{F}|_{X_{t}}, \qquad G_{t}=\mathbf{G}|_{X_{t}} \textnormal{ for } t\in C.$$ 
Assume that, for each $t\in C,$ $$\ho^{i}(\mathcal{H}om(F_{t},G_{t}))=0, \quad i\geq1.$$
Assume that $s:F_{0}\to G_{0}$ is an epimorphism. Then, after possibly shrinking $C,$ the map $s$ extends to an epimorphism $\mathbf{s}:\mathbf{F}\twoheadrightarrow\mathbf{G}.$
\end{lemma}
The above result was originally written for $X=\p3,$ but the proof is valid for any smooth threefold.


We recall, from \cite[Construction 7.2]{ccgm}, that a t'Hooft instanton bundles $\mathcal{E}$ fits in a sequence:
\begin{equation}\label{t'hooft-sequence}
    0\to\mathcal{O}_{\bp3}(-1,1)\to\mathcal{E}\to\mathcal{I}_{X}(1,-1)\to0,
\end{equation}
where $\mathcal{I}_{X}$ is the ideal sheaf of a scheme $$X:=\bigcup_{i=1}^{k-l}C_{i} \cup \bigcup_{j=1}^{l+1}L_{j},$$ 
in which $C_{i}'$s are conics obtained by pull-back of lines in $\p3,$ that do not meet the blow-up point and $L_{j}'$s are fibers of the canonical projection $\bp3\to\p2.$ All of them chosen to be pairwise disjoint. Moreover, such bundles are $\mu$-stable and generically trivial, i.e., $\ho^{1}(L,\mathcal{E}(-2,1)\otimes\mathcal{O}_{L})=0,$ for a general line $L$ in $\bp3,$ and equivalently $\mathcal{E}|_{L}\cong\mathcal{O}_{L}^{\oplus2}.$ Moreover, For each $k,l\geq0$ such that $2k-l\geq2,$ we have an irreducible component $\mathcal{I}^{0}_{\bp3}(k,l)$ of t'Hooft instanton bundles with charge $kH^{2}+lE^{2},$ \cite[Theorem 1.8]{ccgm}.

Before we head to the main result of this section, we shall prove the following 

\begin{lemma}\label{stab-rest-H} Let $\mathcal{E}$ be an instanton bundle, on $\bp3,$ whose restriction $\mathcal{E}|_{E}$ to the exceptional divisor is $\mu$-semi-stable, then its restriction, $\mathcal{E}|_{S},$ to the generic element $S\in|H|,$ in $\bp3,$ is also $\mu$-semi-stable.
In particular, the above assertion is true for every t'Hooft instanton bundle on $\bp3.$ 
\end{lemma}
\begin{proof}
considering the restriction sequence for $S$ twisted by $\mathcal{E}(-1,0)$ and using the fact that $E$ is $\mu_{L}$-stable, one gets the following exact sequence in cohomology:
$$0\to\ho^{0}(\mathcal{E}|_{S}(-1))\to\ho^{1}(\mathcal{E}(-2,0))\to\ho^{1}(\mathcal{E})\to\ho^{1}(\mathcal{E}|_{S})\to0.$$
$S$ is isomorphic to a plane in $\p3,$ and by the Hoppe criteria \cite[Ch.II, Lemma 1.2.5]{oss} it suffices to show that the utmost left term in the sequence above is zero. This is done by considering the sequence 
$$0\to\mathcal{E}(-2,0)\to\mathcal{E}(-2,1)\to\mathcal{E}|_{E}(-1)\to0.$$ Since $\mathcal{E}|_{E}$ is $\mu$-semi-stable \cite[\S 6]{ccgm} then $\ho^{0}(\mathcal{E}|_{E}(-1))$ and the map $$\ho^{1}(\mathcal{E}(-2,0))\to\ho^{1}(\mathcal{E}(-2,1))$$  is injective. Moreover, the instantonic conditions $\ho^{1}\mathcal{E}(-2,1))=0$ implies that $\ho^{1}(\mathcal{E}(-2,0))$ is trivial and hence $\ho^{0}(\mathcal{E}|_{S}(-1))=0.$
 The last assertion follows from the fact that every t'Hooft instanton bundle is trivial on the the generic line, contained in the exceptional divisor \cite[Theorem 1.7]{ccgm}.  
\end{proof}

\begin{corollary}\label{triv-split-bdle}
An instanton bundle $\mathcal{E},$ on $\bp3,$ whose restriction $\mathcal{E}|_{E}$ to the exceptional divisor is $\mu$-semi-stable, is of trivial splitting type on the generic conic $C$ given by the pull-back of a generic line in $\p3.$
\end{corollary}
\begin{proof}
This follows from the $\mu$-semi-stability of $\mathcal{E}_{S},$ for a surface  $S\in |H|,$ containing $C,$ and the Grauert-M\"ulich Thereom \cite[Ch.II, \S2, Corollary 2]{oss}
\end{proof}

\begin{corollary}\label{triv-split-shf}
Let $\mathcal{F}$ be an elementary transform of an instanton bundle $\mathcal{E},$ on $\bp3,$ whose restriction $\mathcal{E}|_{E}$ to the exceptional divisor is $\mu$-semi-stable then it is also of trivial splitting type on the generic conic $C$ given by the pull-back of a generic line in $\p3.$
\end{corollary}
\begin{proof}
The generic conic $C$ does not intersect the curve $\sigma,$ along which the elementary transform is realized, so $\mathcal{F}|_{C}\cong\mathcal{E}|_{C}=\mathcal{O}_{C}^{\oplus 2}.$ The last equality follows from Corollary \ref{triv-split-bdle}.   
\end{proof}

The generic splitting on the conic $C$ can be seen by restricting the sequence $$0\to\mathcal{O}_{\bp3}(-1,1)\to\mathcal{E}\to I_{X}(1,-1)\to0$$ to $C,$ and using the facts that $[C]\cdot(H-E)=H^{3}$ and $C\cap X=\emptyset.$ thus $\mathcal{E}|_{C}=\mathcal{O}_{C}^{\oplus2}.$

Our next result is to show that instanton sheaves obtained via elementary transformations, by lines, from t'Hooft instanton bundles of charge $kH^{2}+lE^{2}$ are in fact in the boundary $\partial\mathcal{I}^{0}_{\bp3}(k',l')$ of locally free t'Hooft instanton component of higher charge $k'H^{2}+l'E^{2}.$ We will say that a union of lines is of type $(d_{1},d_{2}),$ if it is of the form $$\bigcup_{i=1}^{d_{1}}C_{i} \cup \bigcup_{j=1}^{d_{2}}L_{j}.$$

\begin{proposition}\label{main2}
Let $\mathcal{F}$ be an elementary transform of a t'Hooft instanton bundle of charge $kH^{2}+lE^{2},$ on $\bp3,$ with respect to an elementary data $(\Sigma,\mathcal{L},\phi),$ in which $\Sigma$ is a union of lines of type $(d_{1},d_{2}),$ with $d_{1}+d_{2}=r.$ Then $[\mathcal{F}]\in \partial\mathcal{I}^{0}_{\bp3}(k+r,l+d_{2})$ is a smooth point.

\end{proposition}
The proof of this proposition is very similar to \cite[Proposition 7.2 and Proposition 7.3]{JMT1}
\begin{proof}
The proof is by induction on the number of lines $r$ forming $\Sigma$

First we set $r=1,$ and consider $\Sigma$ is a pull-back of the a line in $\p3.$ 
Suppose that we fix a disjoint union of lines $$X:=\bigcup_{i=1}^{k-l}C_{i} \cup \bigcup_{j=1}^{l+1}L_{j},$$ as above, then the sheaves $\mathcal{F}$ given by the extension 
\begin{equation}\label{extension-sheaf}
    0\to\mathcal{O}_{\bp3}(-1,1)\to\mathcal{F}\to\mathcal{I}_{X}(1,-1)\to0
\end{equation}
are classified by $\ext^{1}(\mathcal{I}_{X},\mathcal{O}_{\bp3}(-2,2))$ which is easy to show that it is isomorphic to $\mathbf{V}=\bigoplus_{i=1}^{k-l}\ho^{0}(\mathcal{O}_{C_{i}})\oplus \bigoplus_{j=1}^{l+1}\ho^{0}(\mathcal{O}_{L_{j}}).$ Furthermore, any extension can be identified with its coordinates $(x_{1},\cdots,x_{k-l},y_{1},\cdots,y_{l+1}),$ for $x_{i}\in\ho^{0}(\mathcal{O}_{C_{i}})$ and $y_{j}\in\ho^{0}(\mathcal{O}_{L_{j}}).$ By \cite[Proposition 3.1]{HH}, there exists a flat universal extension parametrized by the affine space
$\mathbf{V}=\mathbb{A}^{k+1},$ and by restricting it to a line $\mathbb{A}^{1}=\{(t,1,\cdots,1)|t\in\mathbb{C}\}$ we get a $1$-dimensional family $\mathbf{F}=\{\mathcal{F}_{t}\}_{t\in\mathbb{A}^{1}}.$ Moreover we have $\mathcal{F}_{t}$ is locally free for $t\neq0.$ But for $t=0,$ one can take the double dual $\mathcal{E}_{0}:=\mathcal{F}_{0}^{\vee\vee}$ which is an instanton sheaf in $\mathcal{I}^{0}_{\bp3}(k-1,l),$ this is because $\mathcal{F},$ is not a locally free instantons, hence its singularity locus is purely one dimensional by Theorem \ref{main1}, in addition, Lemma \ref{v-vanish} implies that the dual is a locally free instanton. Hence $\mathcal{E}_{0}:=\mathcal{F}^{\vee\vee}\cong\mathcal{F}^{\vee}.$ Indeed it will fit in the exact sequence 
\begin{equation}\label{extension-sheaf1}
    0\to\mathcal{O}_{\bp3}(-1,1)\to\mathcal{E}_{0}\to\mathcal{I}_{X_{1}}(1,-1)\to0
\end{equation}
where $$X_{1}:=\bigcup_{i=2}^{k-l}C_{i} \cup \bigcup_{j=1}^{l+1}L_{j},$$
In this way we obtain a morphism $\psi:\mathbb{A}^{1}\to\overline{\mathcal{I}^{0}_{\bp3}(k,l)}$ sending $t$ to the extension $[\mathcal{F}_{t}]$ whose limit is $\mathcal{F}_{0}.$ In other words $[\mathcal{F}_{0}]\in\partial\mathcal{I}^{0}_{\bp3}(k,l).$
Finally the extension \eqref{extension-sheaf} for $\mathcal{F}_{0}$ gives an exact sequence 
\begin{equation}
    0\to\mathcal{F}_{0}\to\mathcal{E}_{0}\to\mathcal{O}_{C_{1}}(1{\rm pt})\to0
\end{equation}
which shows that $\mathcal{F}_{0}$ is an elementary transformation, of $\mathcal{E}_{0},$ whose class $[\mathcal{F}_{0}]$ is a point in $\overline{\mathcal{I}^{0}_{\bp3}(k,l)}.$ it smoothness will be proved below with the help of Lemma \ref{smooth-elementary} .

If $\Sigma$ is a fiber of the projection $\bp3\to\p2,$ then by choosing an other line $\mathbb{A}^{1}=\{(1,1,\cdots,s)|s\in\mathbb{C}\}$ we get another $1$-dimensional family of instanton bundles in $\mathcal{I}^{0}_{\bp3}(k,l-1),$ corresponding to a scheme, whose limit is in $\partial \mathcal{I}^{0}_{\bp3}(k,l)$ and is also smooth. and this completes the proof for  $r=1$

For $r\geq2,$ Let $\Sigma=\Sigma'\cup N,$ where $\Sigma$ is of type $(d'_{1},d'_{2}),$ with $d'_{1}+d'_{2}=r-1$and $N,$ a line of type $(0,1)$ or $(1,0)$, as above:  
We set $Q=Q'\oplus\mathcal{O}_{N}(-{\rm pt})$ and $Q'=\bigoplus_{n=1}^{r-1}\mathcal{O}_{N_{i}}(-{\rm pt}).$ Let $\mathcal{E}$ a t'Hooft instanton bundle and $\mathcal{F}$ its transform with respect to $(\Sigma,\mathcal{L},\phi).$ Then we have successive surjections $\mathcal{E}\stackrel{\phi}{\twoheadrightarrow} Q\otimes\mathcal{O}_{\bp3}(2,-1)\twoheadrightarrow Q'\otimes\mathcal{O}_{\bp3}(2,-1)$
and an exact sequence $$0\to\mathcal{F}\to\mathcal{F}'\stackrel{\phi'}{\to}\mathcal{O}_{N}(1{\rm pt})\to0.$$ By the induction hypothesis, the sheaf $\mathcal{F}'$ is in the boundary of $\partial\mathcal{I}^{0}_{\bp3}(k+r-1,l+d'_{2}).$ for some $d_{i}.$ Thus we can find a flat deformation $\mathbf{F},$ on $U\subset\mathbb{A}^{1},$ for some open affine $U$ containing $0,$ for which $\mathbf{F}_{0}=\mathcal{F}'$ and $\mathbf{F}_{t}$ is locally free and generically trivial (by construction of the t'Hooft instanton bundles above), for all $t\neq0.$ If we set $\mathbf{G}:=p^{\ast}Q',$ where $p:\bp3\times U\to\bp3,$ one has $\ho^{i}(\mathcal{H}om(\mathcal{F}_{t},\mathcal{G}_{t}))=\ho^{i}(\mathcal{O}_{N}(1{\rm pt})^{\oplus2})=0,$ for $i\geq1$ and $t\in U.$ Remark that we included $t=0,$ since the sequence $0\to\mathcal{F}'\to\mathcal{E}\to Q'\otimes\mathcal{O}_{\bp3}(2,-1)\to0$ also implies that $\mathcal{F}'|_{N}=\mathcal{E}|_{N}$ whenever the support of $Q'$ is disjoint from $N$. Finally, by Lemma \ref{surjective-family}, there is a epimorphism $\mathbf{s}:\mathbf{F}\twoheadrightarrow\mathbf{G}$ extending $\phi':\mathcal{F}'\to \mathcal{O}_{N}(1{\rm pt}).$ Thus $[\mathbf{s}_{t}]\in\partial\mathcal{I}^{0}_{\bp3}(k+r,l+d_{2})$ and so is $\mathcal{F}.$ 

\vspace{0.3cm}

It only remains to prove the smoothness: We suppose that the curve $\Sigma$ has the same form of the scheme $X,$ used in \eqref{t'hooft-sequence}, that is, the disjoint union of rational conics $C,$ generic enough to not intersect $X,$ obtained as pull backs of lines in $\p3$ that do not contain the blow-up point. The generic choice of these conics is always possible since thei number is finite and all of them are parametrized by an open subset of $\mathbb{G}r(1,3).$  The lines that we denote by $L$ are fibers of the projection map $\bp3\stackrel{pr}{\to}\p2,$ also generic enough to no intersect $X.$ One shall proceed by treating the situations in which $\Sigma$ is has only one component, namely the two cases  $\Sigma=C$ or $\Sigma=L.$ The general case, of disjoint unions, follows easily by induction. We remind that the locally free t'Hooft instantons $\mathcal{E}$ constructed by Hartshorne-Serre correspondence from the scheme $X,$ used in \eqref{t'hooft-sequence}, are generically trivial \cite[Theorem 1.7]{ccgm} Thus, $\mathcal{E}|_{L}$ is trivial. Recall from Corollary \ref{triv-split-shf} that $\mathcal{E}$ has also a trivial generic splitting on the conic $C,$ that is, $\mathcal{E}|_{C}=\mathcal{O}_{C}^{\oplus2}.$ 

First we start with $\Sigma=C;$ the determinant of its normal sheaf is given by 
$$\det(\mathcal{N}_{\Sigma/\bp3})\cong \mathcal{O}_{\Sigma}(2{\rm pt}).$$

Moreover we have a sequence

\begin{equation}\label{elementary-one-line}
    0\to\mathcal{F}\to\mathcal{E}\to i_{\ast}\mathcal{L}\otimes\mathcal{O}_{\bp3}(2,-1)\to0,
\end{equation}
where $\mathcal{L}=\mathcal{O}_{\Sigma}(-{\rm pt})$ and from the fact that $\mathcal{E}$ is an instanton bundle, we see that the conditions of Lemma \ref{smooth-elementary} are satisfied. Moreover $h^{0}(i_{\ast}\mathcal{L}^{\otimes2}\otimes\mathcal{O}_{\bp3}(4,-2))=h^{0}(\Sigma,\mathcal{O}_{\Sigma}(2{\rm pt}))=3.$ Hence, from Lemma  \ref{smooth-elementary}, it follows that $h^{1}(\mathcal{H}om(\mathcal{F},\mathcal{F}))=8k-4l.$
Now, let us apply $\mathcal{H}om(\bullet,\mathcal{E})$ to the sequence \eqref{elementary-one-line}; one has $\mathcal{H}om(i_{\ast}\mathcal{L}\otimes\mathcal{O}_{\bp3}(2,-1),\mathcal{E})=0,$ and consequently $\mathcal{H}om(\mathcal{F},\mathcal{E})\cong \mathcal{H}om(\mathcal{E},\mathcal{E})$ and $\mathcal{E}xt^{1}(i_{\ast}\mathcal{L}\otimes\mathcal{O}_{\bp3}(2,-1),\mathcal{E})\cong \mathcal{E}xt^{1}(i_{\ast}\mathcal{O}_{\Sigma}, \mathcal{O}_{\bp3})\otimes \mathcal{E}\otimes\mathcal{O}_{\bp3}(-2,1)=0,$ since $\Sigma$ is a curve. 

\vspace{0.4cm}
\textbf{Claim 1:} 
$$\mathcal{H}om(\mathcal{F},i_{\ast}\mathcal{L}\otimes\mathcal{O}_{\bp3}(2,-1))=i_{\ast}\mathcal{O}_{\Sigma}(1{\rm pt})^{\oplus2}\oplus i_{\ast}\mathcal{O}_{\Sigma}(2{\rm pt}).$$

\vspace{0.4cm}

First, we have $\mathcal{H}om(\mathcal{F},i_{\ast}\mathcal{L}\otimes\mathcal{O}_{\bp3}(2,-1))\cong i_{\ast}\mathcal{H}om(i^{\ast}\mathcal{F},\mathcal{O}_{\Sigma}(1{\rm pt})).$ On the other hand, if we apply the composite functor $i^{\ast}(\bullet\otimes i_{\ast}\mathcal{O}_{\Sigma})$ to \eqref{elementary-one-line}, we obtain the sequence
\begin{equation*}
    0\to i^{\ast}Tor^{1}(i_{\ast}\mathcal{L}\otimes\mathcal{O}_{\bp3}(2,-1),i_{\ast}\mathcal{O}_{\Sigma},)\to\mathcal{F}|_{\Sigma}\to\mathcal{E}|_{\Sigma}\to\mathcal{O}_{\Sigma}(1{\rm pt})\to0
\end{equation*}

which can be broken into 

\begin{equation}\label{1piece}
    0\to i^{\ast}Tor^{1}(i_{\ast}\mathcal{L}\otimes\mathcal{O}_{\bp3}(2,-1),i_{\ast}\mathcal{O}_{\Sigma})\to\mathcal{F}|_{\Sigma}\to\mathcal{G}\to0
\end{equation}
and

\begin{equation}\label{2piece}
    0\to\mathcal{G}\to\mathcal{O}_{\Sigma}(-1{\rm pt})\oplus\mathcal{O}_{\Sigma}(1{\rm pt})\to\mathcal{O}_{\Sigma}(1{\rm pt})\to0
\end{equation}

In the above sequence, the generic trivial splitting type of $\mathcal{E}$ has been used, {\rm i. e.}, $\mathcal{E}_{\Sigma=C}=\mathcal{O}_{\Sigma}(-1{\rm pt})\oplus\mathcal{O}_{\Sigma}(1{\rm pt}).$  From \eqref{2piece} one can see that $\mathcal{G}=\mathcal{O}_{\Sigma}(-1{\rm pt}).$ Remark that the sheaf $i^{\ast}Tor^{1}(i_{\ast}\mathcal{L}\otimes\mathcal{O}_{\bp3}(2,-1),i_{\ast}\mathcal{O}_{\Sigma})$ is in fact equal to $\mathcal{N}^{\vee}_{\Sigma/\bp3}\otimes\mathcal{O}_{\Sigma}(1{\rm pt})\cong\mathcal{O}_{\Sigma}^{\oplus2}.$ From \eqref{1piece}, we conclude that $\mathcal{F}|_{\Sigma}=\mathcal{O}_{\Sigma}^{\oplus2}\oplus \mathcal{O}_{\Sigma}(-1{\rm pt}),$ and the claim follows.

\vspace{0.4cm}

\textbf{Claim 2:} $\mathcal{E}xt^{1}(\mathcal{F},\mathcal{E})=i_{\ast}\mathcal{O}_{\Sigma}\oplus i_{\ast}\mathcal{O}_{\Sigma}(2{\rm pt})$

\vspace{0.4cm}

By applying the functor $\mathcal{H}om(\bullet,\mathcal{E})$ on \eqref{elementary-one-line}, we get
\begin{align}
\mathcal{E}xt^{1}(\mathcal{F},\mathcal{E})=&\mathcal{E}xt^{2}(i_{\ast}\mathcal{L}\otimes\mathcal{O}_{\bp3}(2,-1),\mathcal{E}) \notag \\ &=\mathcal{E}xt^{2}(i_{\ast}\mathcal{O}_{\Sigma}(-1{\rm pt}),\mathcal{O}_{\bp3})\otimes\mathcal{E}(-2,1) \notag \\
&=\mathcal{E}xt^{2}(i_{\ast}\mathcal{O}_{\Sigma},\mathcal{O}_{\bp3})\otimes i_{\ast}\mathcal{O}_{\Sigma}(1{\rm pt})\otimes\mathcal{E}(-2,1) \notag \\
&=i_{\ast}\mathcal{O}_{\Sigma}\oplus i_{\ast}\mathcal{O}_{\Sigma}(2{\rm pt}) \notag
\end{align}

Where the last equality is, again, obtained by using$$\mathcal{E}xt^{2}(i_{\ast}\mathcal{O}_{\Sigma},\mathcal{O}_{\bp3})\cong i_{\ast}\omega_{\Sigma}\otimes\mathcal{O}_{\bp3}(4,-2)\cong i_{\ast}\mathcal{O}_{\Sigma}(2{\rm pt}),$$ 
and  $\mathcal{E}|_{\Sigma}=\mathcal{O}_{\Sigma}(-1{\rm pt})\oplus\mathcal{O}_{\Sigma}(1{\rm pt}).$ 

\vspace{0.4cm}

\textbf{Claim 3:} There is a long exact sequence 
\begin{align}\label{long-ext-sequence}
    0\to& i_{\ast}\mathcal{L}^{\otimes2}\otimes\mathcal{O}_{\bp3}(4,-2) \to\mathcal{H}om(\mathcal{F},i_{\ast}\mathcal{L}\otimes\mathcal{O}_{\bp3}(2,-1))\to\mathcal{E}xt^{1}(\mathcal{F},\mathcal{F})\\ 
    &\to\mathcal{E}xt^{1}(\mathcal{F},\mathcal{E})\to\mathcal{E}xt^{1}(\mathcal{F},i_{\ast}\mathcal{L}\otimes\mathcal{O}_{\bp3}(2,-1))\to0 \notag
\end{align}

\vspace{0.4cm}

If We apply the functor $\mathcal{H}om(\mathcal{F},\bullet)$ to \eqref{elementary-one-line}, then we get 

\begin{align}
    0\to&\mathcal{H}om(\mathcal{F},\mathcal{F}) \to\mathcal{H}om(\mathcal{F},\mathcal{E})\to\mathcal{H}om(\mathcal{F},i_{\ast}\mathcal{L}\otimes\mathcal{O}_{\bp3}(2,-1))\\ 
    &\to\mathcal{E}xt^{1}(\mathcal{F},\mathcal{F})\to\mathcal{E}xt^{1}(\mathcal{F},\mathcal{E})\to\mathcal{E}xt^{1}(\mathcal{F},i_{\ast}\mathcal{L}\otimes\mathcal{O}_{\bp3}(2,-1))\to0, \notag
\end{align}
since $\mathcal{E}xt^{2}(\mathcal{F},\mathcal{F})=0,$ as we saw in the proof of Lemma \ref{smooth-elementary}. Also, from the proof of Proposition \ref{hom-quotients}, one has $\mathcal{H}om(\mathcal{F},\mathcal{E})\cong\mathcal{H}om(\mathcal{E},\mathcal{E}),$ and the claim follows.

\vspace{0.4cm}

\textbf{Claim 4:} $\mathcal{E}xt^{1}(\mathcal{F},i_{\ast}\mathcal{L}\otimes\mathcal{O}_{\bp3}(2,-1))\cong i_{\ast}\mathcal{O}(2{\rm pt})$

\vspace{0.4cm}

Now, if we apply $\mathcal{H}om(\bullet, i_{\ast}\mathcal{L}\otimes\mathcal{O}_{\bp3}(2,-1))$ on \eqref{elementary-one-line}, we obtain

\begin{align}
    \mathcal{E}xt^{1}(\mathcal{F},i_{\ast}\mathcal{L}\otimes\mathcal{O}_{\bp3}(2,-1))&=\mathcal{E}xt^{2}(i_{\ast}\mathcal{L}, i_{\ast}\mathcal{L}) \\
    &\cong i_{\ast}\omega_{\Sigma}\otimes\mathcal{O}_{\bp3}(4,-2)=i_{\ast}\mathcal{O}(2{\rm pt}). \notag 
\end{align}

Then by using the sequence \eqref{long-ext-sequence} and the proved Claims 1-2 one has:

\begin{align}
    0\to& i_{\ast}\mathcal{O}_{\Sigma}(2{\rm pt}) \to i_{\ast}\mathcal{O}_{\Sigma}(1{\rm pt})^{\oplus2}\oplus i_{\ast}\mathcal{O}_{\Sigma}(2{\rm pt})\to\mathcal{E}xt^{1}(\mathcal{F},\mathcal{F})\\ 
    &\to i_{\ast}\mathcal{O}_{\Sigma}(1{\rm pt})^{\oplus2} \to i_{\ast}\mathcal{O}(2{\rm pt})\to0 \notag
\end{align}

Again, by breaking this long sequence into short exact ones, one obtains the isomorphism $$\mathcal{E}xt^{1}(\mathcal{F},\mathcal{F})\cong i_{\ast}\mathcal{O}_{\Sigma}(1{\rm pt})^{\oplus2}\oplus i_{\ast}\mathcal{O}_{\Sigma},$$ thus $$h^{0}(\mathcal{E}xt^{1}(\mathcal{F},\mathcal{F}))=5, \quad h^{1}(\mathcal{E}xt^{1}(\mathcal{F},\mathcal{F}))=0,$$ Finally from Proposition \ref{smooth-elementary} it follows that ${\rm ext}^{1}(\mathcal{F},\mathcal{F})=5+8k-4l-3+3=8(k+1)-4l-3.$ 

Recall that $\mathcal{I}^{0}_{\bp3}(k+1,l)$ has dimension $8(k+1)-4l-3.$ Moreover $[\mathcal{F}] \in \partial\mathcal{I}^{0}_{\bp3}(k+1,l)$ and $\dim(T_{[\mathcal{F}]}\mathcal{I}^{0}_{\bp3}(k+1,l))={\rm ext}^{1}(\mathcal{F},\mathcal{F})=8(k+1)-4l-3.$ It follows that there is no jump in the dimension of the tangent at $[\mathcal{F}]$ and therefore it is a smooth point in the closure of $\mathcal{I}^{0}_{\bp3}(k+1,l).$ This completes the proof for $\Sigma= C$.

In the case of $\Sigma= L,$ one makes the following changes 

\begin{align}
&\mathcal{N}_{\Sigma/\bp3}=\mathcal{O}_{\Sigma}(1{\rm pt})\oplus\mathcal{O}_{\Sigma}(-1{\rm pt}),\hspace{0,4cm}\det(\mathcal{N}_{\Sigma/\bp3})=\mathcal{O}_{\Sigma},  \notag \\
&\mathcal{E}|_{\Sigma}=\mathcal{O}_{\Sigma}^{\oplus2} ,\hspace{0,2cm} \textnormal{ and } ,\hspace{0,2cm} i_{\ast}\mathcal{L}\otimes\mathcal{O}_{\bp3}(2,-1))=\mathcal{O}_{\Sigma}; \notag
\end{align}

then one can check that $$\mathcal{E}xt^{1}(\mathcal{F},\mathcal{F})\cong i_{\ast}\mathcal{O}_{\Sigma}(1{\rm pt})\oplus i_{\ast}\mathcal{O}_{\Sigma}\oplus\mathcal{O}_{\Sigma}(-1{\rm pt}),$$
hence $$h^{0}(\mathcal{E}xt^{1}(\mathcal{F},\mathcal{F}))=3, \quad h^{1}(\mathcal{E}xt^{1}(\mathcal{F},\mathcal{F}))=0,$$ Finally from Proposition \ref{smooth-elementary} it follows that ${\rm ext}^{1}(\mathcal{F},\mathcal{F})=3+8k-4l-3+1=8k-4l+1=8(k+1)-4(l+1)-3.$


In order to finish the smoothability proof we proceed by induction, as in \cite[\S3.1]{faenzi}. Let $\ft$ be a deformation of a torsion free instanton $\mathcal{F}$ sheaf obtained as an elementary transformation of a locally free t'Hooft instanton
\begin{equation}\label{ft}
0\to\mathcal{F}\to\mathcal{E}\to\qsigma\to0
\end{equation}

where 
\begin{equation}
\qsigma=\left\{\begin{array}{cc} \mathcal{O}_{L} & \Sigma=L \\ \mathcal{O}_{C}(1{\rm pt}) & \Sigma=C\end{array}\right.
\end{equation}

If $\ft$ is not locally free, then it fits in the short exact sequence: 

\begin{equation}\label{double-d}
0\to\ft\to\ft^{\vee\vee}\to\calq\to0.
\end{equation}

where the support of the quotient sheaf $\calq$ has codimension $\geq2.$ 
First, we want to determine possible sheaves $\calq$ that can appear in \eqref{double-d} by computing its Chern classes. Clearly the class $c_{1}(\calq)$ is zero. The second Chern class $c_{2}(\calq)$ can be written as $s\cdot H^{2}+ t\cdot E^{2},$ with $ s,t\leq0$ as it is minus the class of the support ${\rm Supp}(\calq),$ of $\calq,$ in the Chow ring. Indeed, ${\rm Supp}(\calq)$ is an effective $1$-dimensional cycle, in $\bp3,$ and its second Chern class is twist invariant by line bundles $\mathcal{O}_{\bp3}(p,q).$  

Notice that the Chern class $c_{3}(\ft^{\vee\vee})=m\cdot H^{3},$ for $m\geq0$ is also invariant by twist, since $\ft^{\vee\vee}$ is reflexive. It also turns out that this class is zero as we apply the Hirzebruch-Riemann-Roch Theorem \eqref{euler}; indeed one has 
\begin{equation}   
\chi(\calq(p,q))=\frac{m}{2}-s\cdot(p+2)-t\cdot(q-1)
\end{equation}

By semi-continuity we have that $\ho^{i}(\ft(-2,1))=0,$ for $0\leq i\leq3.$ Moreover, the vector space $\ho^{0}(\ft^{\vee\vee}(-2,1))$ is trivial, since the double dual $\ft^{\vee\vee}$ of the $\mu_{\mathcal{L}}$semi-stable sheaf $\calf$ is also $\mu_{\mathcal{L}}$semi-stable. Hence, by choosing $p=-2$ and $q=1,$ we have that $\chi(\calq(-2,1))=\frac{m}{2}=-h^{1}(\calq(-2,1))\leq0.$ But $m\geq0,$ thus is should be zero and it follows that $\ft^{\vee\vee}$ is locally free, by {\rm item (ii)} of Lemma \ref{intg-condition}. Furthermore, $\ft^{\vee\vee}$ is a t'Hooft instanton.

We recall the reader of the following 
\begin{theorem}(Fujita's vanishing theorem\cite[\S 1.4.D]{lazar})\label{fujita}
Let $X$ be a complex projective scheme and let $D$ be an ample (integral) divisor on $X.$ Given any coherent sheaf $\mathcal{G}$ on $X,$ there exists an integer $d(D,\mathcal{G})$ such that $$\ho^{i}(X,\mathcal{G}\otimes\mathcal{O}_{X}(d\cdot D + D'))=0, \quad\forall i>0, d\geq d(D,\mathcal{G})$$ and for any nef divisor $D'.$
\end{theorem}
In fact, the integer $d(D,\mathcal{G})$ is independent of the nef divisor $D'.$ 

\begin{corollary} \label{fujita-rk-2}
Let $\mathcal{G}$ be a rank $2$ locally free sheaf with $c_{1}(\mathcal{G})=0$ on a fano variety $X$ of index $i_{X}$ and $h:\frac{-K_{X}}{i_{X}}.$ Then $$\ho^{i}(X,\mathcal{G}\otimes\mathcal{O}_{X}(-(d+i_{X})\cdot h - D'))=0, \quad\forall i<\dim X, d\geq d(D,\mathcal{G})$$ and for any nef divisor $D'.$
\end{corollary}
\begin{proof} Since $\mathcal{G}$ is a rank $2$ bundle with trivial first Chern class, then $\mathcal{G}\cong\mathcal{G}^{\vee}.$ Thus, the corollary follows by applying Serre duality and Fujita's vanishing theorem
\end{proof}
In our case, $i_{X}=2,$ $h=2H-E,$ and by Bertini's theorem the generic element of the linear system $|h=2H-E|$ is integral. A divisor $D'=a\cdot H+b\cdot E$ is nef if, and only if $a\geq0$ and $a+b\geq0.$

Let us choose the twist by $\mathcal{O}_{\bp3}(-(d+2)\cdot(2H-E)+(a\cdot H+b\cdot E))=\mathcal{O}_{\bp3}(-2(d+2)-a,(d+2)-b),$ which is of the form $\mathcal{O}_{X}(-(d+i_{X})\cdot h - D').$ Thus, by Corollary \eqref{fujita-rk-2} we shall have $\ho^{i}(\mathcal{E}(-2(d+2)-a,(d+2)-b))=0,$ and $\ho^{i}(\ft^{\vee\vee}(-2d,d))=0,$ $\forall i<3,$ $d\geq d(2H-E,\ft^{\vee\vee}), d(2H-E,\mathcal{E})$ and for any nef divisor $D'=a\cdot H+b\cdot E,$ for the instanton bundle $\mathcal{E}$ and the locally free sheaf $\ft^{\vee\vee}$ in sequences  \eqref{ft} and \eqref{double-d}, respectively. On the other hand one has  $h^{2}(\ft(-2(d+2)-a,(d+2)-b))\leq h^{2}(\mathcal{F}(-2(d+2)-a,(d+2)-b))=h^{1}(\qsigma(-2(d+2)-a,(d+2)-b)),$ by semi-continuity. Consequently, we obtain the inequalities

\begin{equation} \label{bound-s}  
\begin{array}{cc} s\cdot(+2)+t\cdot(q+1)\leq -(p+2)& \textnormal{ for } \Sigma=C; \\ 
s\cdot(p+2)+t\cdot(q+1)\leq -(p+q+1) &  \textnormal{ for } \Sigma=L. \end{array}
\end{equation}
for $p=-2(d+2)-a, q=(d+2)-b.$

The fact that the support of $\calq$ is effective means that $s,t\leq0.$ For the first case $ \Sigma=C$ the inequality \eqref{bound-s} returns $$s\geq-1+t\frac{d-b+3}{2d+a+2}$$
and by choosing $d=b$ and $d>>0,$ one obtains $s=-1$ or $s=0.$ 
\begin{itemize}
\item[(i)] For $s=0,$ inequality \eqref{bound-s} gives $$t\geq t\frac{d+3}{-b} + \frac{2d+a+2}{-b},$$ so by taking $b>>d$
and $d>>0$ we get $t\geq0.$ Hence $t=0.$ In this case the support of $\calq$ is empty and $\ft$ is locally free.

\item[(ii)] For $s=-1,$ inequality \eqref{bound-s} reads $$t\cdot(d-b+3)\leq0$$ and by taking $b>>d,$ it follows that $t=0,$ hence $\calq=\mathcal{O}_{C}(1).$ To conclude the proof in this case, we shall show that one cannot fill a whole component with such a deformation; this follows by a dimensional argument. Indeed, one has to pick $\ft^{\vee\vee}$ in a moduli of $(k,l)$-instantons and this corresponds to $8k-4l-3,$ besides the choice of the support $C$ that corresponds to a point in an open subset of one of the two components of the Hilbert scheme of lines. This component is the grassmannian of lines in $\p3,$  which is a $4$-dimensional family. In addition, we have to chose a surjection $\ft^{\vee\vee}\to\mathcal{O}_{C}(1)$ in  $$\phom(\ft^{\vee\vee},\mathcal{O}_{C}(1))=\ph^{0}(\ft^{\vee\vee}\otimes\mathcal{O}_{C}(1)).$$ By applying Corollary \ref{triv-split-bdle}, we have that $dim\ph^{0}(\ft^{\vee\vee}\otimes\mathcal{O}_{C}(1))=3.$ Summing up, the deformation $\ft$ sits into a family of dimension $(8k-4l-3)+4+3=8k-4l+4\leq8(k+1)-4l-3=\dim T_{[\mathcal{F}]}\mathcal{I}^{0}_{\bp3}(k+1,l).$ Therefore, apart from such deformations $\ft,$ $\mathcal{F}$ is deformable to a locally free $(k+1,l)$instanton, hence smoothable.

\end{itemize}

For the  case $ \Sigma=L$ the inequality \eqref{bound-s} returns $$s\geq t\frac{d-b+3}{2d+a+2}-\frac{d+a+b+1}{2d+a+2}$$
and by choosing $a>>d$ and $d>>0,$ we obtain $s=-1$ or $s=0,$ again. 
\begin{itemize}
\item[(i)] For $s=-1,$ inequality \eqref{bound-s} gives $$ \left\{\begin{array}{cc}t\geq\frac{b-d-1}{d-b+3}; & b>d, d>>0 \\ & \\ t\leq\frac{b-d-1}{d-b+3}; & d>b, b>>0.\end{array}\right.\Longrightarrow\left\{\begin{array}{cc}t\geq-1; & b>d, \\ & \\ t\leq-1 & d>b.\end{array}\right.$$ Hence $t=-1$and it follows that $\calq=\mathcal{O}_{L}.$ In this case, one has to select a  $(k,l)$-instantons $\ft^{\vee\vee}$ and the support $L$ inside one of the two component of Hilbert scheme of lines in $\bp3,$ both of which are $2$-dimensional. The other component parametrizes lines $l_{E}$ in the exceptional divisor \cite[Remark 6.1]{ccgm}. In addition, we choose a surjection $\ft^{\vee\vee}\to\mathcal{O}_{L}$ in  $\phom(\ft^{\vee\vee},\mathcal{O}_{L})$ By applying the general trivial splitting of t'Hooft instantons on lines $L$ we get $\dim\ph^{0}(\ft^{\vee\vee}\otimes\mathcal{O}_{L})=1.$ Summing up, the deformation $\ft$ sits into a family of dimension $(8k-4l-3)+2+1=8k-4l\leq8(k+1)-4l-3=\dim(T_{[\mathcal{F}]}\mathcal{I}^{0}_{\bp3}(k+1,l).$ Again, aside from such deformations $\ft,$ $\mathcal{F}$ is deformable to a locally free $(k+1,l)$instanton. Thus, also smoothable in this case.

\item[(ii)] For $s=0,$ inequality \eqref{bound-s} reads $$ \left\{\begin{array}{cc}t\geq\frac{d+a+b+1}{d-b+3}; & b>d, d>>0 \\ & \\ t\leq\frac{d+a+b+1}{d-b+3}; & d>b, b>>0.\end{array}\right.\Longrightarrow\left\{\begin{array}{cc}t\geq1; & b>d, \\ & \\ t\leq1 & d>b.\end{array}\right.$$ Hence $t=1.$ Since the support of $\calq$ is effective, it follows that it is empty in this case, hence  $\ft$ is again locally free.

\end{itemize}

\end{proof}

By combining smoothability with Corollary \ref{stable-existence} we get the following 

\begin{corollary}
There exist $\mu_{L}$-stable rank $2$ instanton bundles of any charge $kH^{2}+lE^{2}$ .
\end{corollary}

\vspace{0.2cm}

{\small

Abdelmoubine Amar Henni   

Departamento de Matem\'atica MTM - UFSC

Campus Universit\'ario Trindade CEP 88.040-900 Florian\'opolis-SC, Brazil

e-mail: henni.amar@ufsc.br   

}

\end{document}